\documentclass{article}
\pdfoutput=1
\usepackage[utf8]{inputenc}
\usepackage{enumerate}
\usepackage{amsfonts,amsmath,amssymb,amsthm,bbm,bm}
\usepackage{mathtools}
\usepackage{xcolor}
\usepackage{soul}
\setstcolor{red}

\usepackage{graphicx}
\usepackage{caption}
\usepackage{subcaption}

\usepackage[style=alphabetic,backend=biber, sorting = nyt]{biblatex}
\usepackage{hyperref}
\hypersetup{colorlinks}
\usepackage[colorinlistoftodos]{todonotes}

\newcommand{\calP}{\mathcal{P}}
\newcommand{\calA}{\mathcal{A}}
\newcommand{\calN}{\mathcal{N}}

\newcommand{\calD}{\mathcal{D}}
\newcommand{\R}{\mathbb{R}}
\newcommand{\N}{\mathbb{N}}
\newcommand{\E}{\mathbb{E}}

\renewcommand{\P}{\mathbb{P}}

\newcommand{\WD}{\mathcal{W}}
\newcommand{\AWD}{\mathcal{AW}}

\newcommand{\cpl}{\text{Cpl}}
\newcommand{\bccpl}{\text{Cpl}_{\text{bc}}}

\usepackage{thmtools}
\usepackage{hyperref}
\usepackage{cleveref}
\newtheorem{theorem}{Theorem}[section]

\newtheorem{lemma}[theorem]{Lemma}

\newtheorem{setting}[theorem]{Setting}
\theoremstyle{remark}
\newtheorem{remark}[theorem]{Remark}

\theoremstyle{definition}
\newtheorem{definition}[theorem]{Definition}
\newtheorem{example}[theorem]{Example}

\definecolor{newblue}{rgb}{0.0, 0.0, 0.9}
\newcommand{\add}[1]{{\color{black}#1}}

\title{Convergence of Adapted Empirical Measures on $\mathbb{R}^{d}$}

\author{Beatrice Acciaio\thanks{Department of Mathematics, ETH Z\"{u}rich, Switzerland.\newline \emph{beatrice.acciaio@math.ethz.ch}, \emph{songyan.hou@math.ethz.ch}} \, ~and~ Songyan Hou\footnotemark[1]}
\date{\today }

\begin{document}

\maketitle

\begin{abstract}
We consider empirical measures of $\R^{d}$-valued stochastic process in finite discrete-time. We show that the adapted empirical measure introduced in the recent work \cite{backhoff2022estimating} by Backhoff et al. in compact spaces can be defined analogously on $\R^{d}$, and that it converges almost surely to the underlying measure under the adapted Wasserstein distance. Moreover, we quantitatively analyze the convergence of the adapted Wasserstein \add{distance} between those two measures. We establish convergence rates of the expected error as well as the deviation error under different moment conditions. \add{Under suitable integrability and kernel assumptions, we recover the optimal convergence rates of both expected error and deviation error.} Furthermore, we propose a modification of the adapted empirical measure with \add{projection} on a non-uniform grid, which obtains the same convergence rate but under weaker assumptions. \\

\end{abstract}

\section{Introduction}\label{sec:introduction}
Empirical measures analysis addresses the fundamental question of how well empirical observations approximate the underlying distribution. The quality of the approximation is measured by how close the empirical value approximates the true value that we are interested in (e.g. expected utility). Precisely stated, it studies the convergence in functional values under the empirical measure to the ones under the true, underlying measure. For many problems, the functional is of the form $\int f(x)\mu(dx)$, where $f$ is a real-valued continuous function and $\mu$ is a probability measure. Under 
mild integrability assumptions, such functional is continuous with respect to the Wasserstein distance, an important metric intensively studied in statistics theories \cite{fournier2015rate} as well as adopted in machine learning applications \cite{adler2018banach}. Therefore, the convergence of empirical measures is classically studied under Wasserstein distance. However, in stochastic finance, many important problems such as pricing and hedging problems, optimal stopping problems, and utility maximization problems, are actually not continuous with respect to Wasserstein distance.  To wit, two stochastic models can be arbitrarily close to each other under Wasserstein distance, while the values of the mentioned optimization problems under the two models can be quite far from each other. 

These problems, as well as many other stochastic optimization problems in a dynamic framework, are instead Lipschitz continuous with respect to a more strict distance, called adapted
Wasserstein distance, which takes the temporal structure of stochastic processes into account; see \cite{backhoff2020adapted}. This Lipschitz property implies that the adapted Wasserstein distance is strong enough to guarantee the robustness of path-dependent problems. On the other hand, it is so strong that even the empirical measures do not converge to the corresponding underlying measure under it; see \cite{pflug2016empirical}. Intuitively speaking, this is because the empirical measure is so ``singular" that it can not explain the conditional laws. A remedy to this undesirable effect is to modify the empirical measure, so that the new measure has a nontrivial temporal structure and that it converges under adapted Wasserstein distance. In this spirit, Pflug and Pichler in \cite{pflug2016empirical} deploy smoothing convolutions on paths and prove that the convoluted measures converge. Alternatively, Backhoff et al. in \cite{backhoff2022estimating} project paths on a grid and prove that the empirical measures of the projected paths converge. Both works prove that modified measures converge under adapted Wasserstein distance, though they are both limited to the case when the underlying measure is compactly supported. This condition is not satisfied by many important models in stochastic finance, e.g. the discretized Black-Scholes model, the Autoregressive model, the Fama-French model, etc. 

In the present paper, we drop the assumption of compactness, and investigate the case of general, finite discrete-time stochastic processes which take values on $\R^{d}$. We stress the importance of being able to deal with measures with unbounded support, as this is always the case for stochastic models used in finance and insurance applications and beyond.
We define the adapted empirical measure on $\R^{dT}$ analogously to the one in \cite{backhoff2022estimating}, and establish convergence results under adapted Wasserstein distance.

The first main contribution of the present paper is the generalization of all convergence results in \cite{backhoff2022estimating} beyond the compactly supported constraint. We first show that adapted empirical measures converge under adapted Wasserstein distance almost surely as long as the underlying measure is integrable. Then, we compute the convergence rates of the expected error as well as of the deviation error under the uniform finite moment condition and the uniform Lipschitz kernel assumption. Under the finite moment condition, we prove polynomial convergence rates of the expected error as well as of the deviation error depending on the order of the finite moment. In particular, if the underlying measure $\mu \in \mathcal{P}(\R^{dT})$ satisfies $p$-th moment conditions with $p > \frac{d}{d-1}$, we recover the optimal rates of the expected error, which are known for the classical empirical measure with respect to Wasserstein distance; see \cite{fournier2015rate}.
Under the exponential moment condition, we prove the sharp optimal convergence rates of the expected error as well as the deviation error. Indeed, this recovers the optimal expected error rates and exponential deviation rates in \cite{backhoff2022estimating}.

Another main contribution of the article is the introduction of the non-uniform adapted empirical measure. We define a non-uniform grid on $\R^{dT}$, which is denser around the origin and coarser at the distance. Then we project paths on this non-uniform grid and refer to the empirical measure of the projected paths as the non-uniform adapted empirical measure. We show that all convergence results mentioned above are valid for the non-uniform adapted empirical measure as well. Moreover, in this case we can obtain the same polynomial convergence rate of the expected error but under a milder moment assumption, independent of $T$ and decaying with the dimension $d$.

The results obtained in this paper are of relevance for the wide range of applications of the adapted empirical measures on $\R^d$. In statistical learning theory, our results generalize the consistency result of empirical Wasserstein correlation coefficient in \cite{Wiesel2022MeasuringAW}. In machine learning, our results contribute a primal numerical scheme evaluating the adapted Wasserstein distance as for the discriminator of generative adversarial models in \cite{Xu2020COTGANGS,Klemmer2022SPATEGANIG}. In mathematical finance, our results provide finite samples guarantee of distributionally robust evaluation, such as superhedging prices, expected risks, etc.; see \cite{Han2022DistributionallyRR, Obj2021RobustEO}.\\

\noindent{\bf Organization of the paper.}
We conclude the Introduction with a review of the related literature.
Then, in Section \ref{sec:results}, we introduce the uniform and the non-uniform adapted empirical measures and state our main convergence results. In Section \ref{sec:preparation} we provide lemmas which will be used in the proofs of the main results. In Section \ref{sec:moment_estimate} we prove convergence rates of the expected error. In Section \ref{sec:concentration1} we prove polynomial convergence rates of the deviation error under finite moment conditions. In Section \ref{sec:concentration2} we prove exponential convergence rates of the deviation error under finite exponential moment conditions. In Section \ref{sec:asconvergence} we prove almost sure convergence under an integrability assumption.

\subsection{Related literature.}
\subsubsection*{Empirical measure under Wasserstein distance.}
In the literature, much effort has been devoted to the analysis of empirical measures under Wasserstein distance, see e.g. \cite{bolley2007quantitative,dedecker2015deviation,fournier2015rate,gozlan2007large,lei2020convergence,boissard2011simple,dereich2013constructive,boissard2014mean}. Concerning moment estimation, convergence rate results can be found in \cite{boissard2014mean} based on iterative trees, and in \cite{dereich2013constructive} based on a so-called Pierce-type estimate. Later, Fournier et al. \cite{fournier2015rate} prove the sharp convergence result (as far as general laws are considered) by leveraging the Pierce-type estimate in \cite{dereich2013constructive}. 

Regarding the concentration inequalities, most papers concern 0-concentration and only a few discuss the mean concentration. For the 0-concentration inequality, the works \cite{boissard2011simple}, \cite{bolley2007quantitative} and \cite{gozlan2007large} establish the exponential rate under the assumption of transportation-entropy inequality. Fournier et al. in \cite{fournier2015rate} prove the concentration inequality under moment conditions or exponential moment conditions. For the mean-concentration inequality, Dedecker et al. in \cite{dedecker2015deviation} prove it for the Wasserstein-1 distance, using the Lipschitz formulation of the duality. Lei \cite{lei2020convergence} considers the mean-concentration inequality in unbounded functional space and also discusses the moment estimate with rate-optimal upper bounds for functional data distributions.

\subsubsection*{Empirical measure under adapted Wasserstein distance.}
Pflug-Pichler \cite{pflug2016empirical} first notice that the empirical measure of stochastic processes does not converge to the true underlying measure under adapted Wasserstein distance. To amend to this, they introduce a type of convoluted empirical measure which is a convolution of a smooth kernel with the empirical measure. They prove that this convoluted empirical measure converges to the true underlying measure under adapted Wasserstein distance if the underlying measure is compactly supported  with sufficiently regular density bounded away from $0$ and uniform Lipschitz kernels.  Recently, this convoluted empirical measure is applied in \cite{glanzer2019incorporating} incorporating statistical model error into robust prices of contingent claims.

More recently, Backhoff et al. in \cite{backhoff2022estimating} introduce the adapted empirical measure which projects paths on a grid before taking the empirical measure, and establish multiple convergence results. Again under the assumption that the underlying measure is compactly supported, they prove the almost sure convergence of the adapted empirical measure. Under further uniform Lipschitz kernel assumptions, they recover the same moment estimate convergence rate as the one for Wasserstein distance in \cite{fournier2015rate}, which is sharp. Moreover, they prove the exponential concentration inequality under Lipschitz kernel assumptions. Our paper generalizes their result to measures on $\R^{d}$, and introduces the non-uniform adapted empirical measure to weaken the moment assumption. 

\section{Settings and main results}\label{sec:results}
Throughout the paper, we let \add{$d \in \N$} be the dimension of the state space and $T \geq 2$ be the number of time steps. We consider finite discrete-time paths $x = (x_1,\dots,x_T) \in \R^{dT}$, where $x_{t} \in \R^{d}$ represents the value of the path at time $t = 1,\dots, T$. We equip $\R^{dT}$ with a sum-norm $\Vert \cdot \Vert \colon \R^{dT} \to \R$ defined by $\Vert x \Vert = \sum_{t = 1}^{T} \Vert x_{t} \Vert_{\R^{d}}$ for simplicity, but without loss of generality since all norms are equivalent on $\R^{dT}$. Let $\mathcal{P}(\R^{dT})$ be the space of canonical Borel probability measures on $\R^{dT}$, and let $\mu, \nu \in \mathcal{P}(\R^{dT})$.
\begin{definition}[Wasserstein distance]
    For $p \geq 1$, the $p$-th order \textit{Wasserstein distance} $\WD_{p}(\cdot,\cdot)$ on $\mathcal{P}(\R^{dT})$ is defined by 
    \begin{equation*}
        \WD_{p}^{p}(\mu,\nu) = \inf_{\pi \in \cpl(\mu,\nu)}\int\sum_{t=1}^{T}\big\lVert x_{t}-y_{t}\big\rVert^p_{\R^{d}}\pi(dx,dy),
    \end{equation*}
    where $\cpl(\mu,\nu)$ denotes the set of couplings between $\mu$ and $\nu$, that is, probabilities in $\mathcal{P}(\R^{dT}\times \R^{dT})$ with first marginal $\mu$ and second marginal $\nu$. For the first order Wasserstein distance, i.e. when $p=1$, we simply use the notation $\WD(\cdot,\cdot)$.
\end{definition}

In order to take the temporal time structure of stochastic processes into account, it is necessary to restrict ourselves to those couplings which not only match marginal laws but also match conditional laws. We adopt the notation $x_{1:t}=(x_1,...,x_t)$, for $t=1,...,T$.
For $\mu \in \mathcal{P}(\R^{dT})$, we denote the up to time $t$ marginal of $\mu$ by $\mu_{1:t}$, and the kernel (disintegration) of $\mu$ by $\mu_{x_{1:t}}$, so the following holds:
\begin{equation*}
    \mu(dx_{t+1}) = \int_{\R^{dt}}\mu_{x_{1:t}}(dx_{t+1}) \cdot \mu_{1:t}(dx_{1:t}).
\end{equation*}
Similarly, we denote the up to time $t$ marginal of $\pi\in\cpl(\mu,\nu)$ by $\pi_{1:t}$, and the kernel (disintegration) of $\pi$ by $\pi_{x_{1:t}, y_{1:t}}$, so the following holds:
\begin{equation*}
    \pi(dx_{t+1},dy_{t+1}) = \int_{\R^{dt} \times \R^{dt}}\pi_{x_{1:t},y_{1:t}}(dx_{t+1},dy_{t+1})\cdot \pi_{1:t}(dx_{1:t},dy_{1:t}).
\end{equation*} 
Intuitively, we restrict our attention to couplings $\pi \in \cpl(\mu,\nu)$ such that the conditional law of $\pi$ is still a coupling of the conditional laws of $\mu$ and $\nu$, that is, 
\begin{equation}\label{eq:bicausal}
    \pi_{x_{1:t},y_{1:t}} \in \cpl\big(\mu_{x_{1:t}}, \nu_{y_{1:t}}\big).
\end{equation}
Such couplings are called
bi-causal \cite{lassalle2018causal}, and denoted by $\bccpl(\mu,\nu)$.
The causality constraint can be expressed in different equivalent ways, see e.g. \cite{backhoff2017causal,acciaio2020causal} in the context of transport, and \cite{bremaud1978changes} in the filtration enlargement framework. Roughly, in a causal transport, for every time $t$, only information on the $x$-coordinate up to time $t$ is used to determine the mass transported to the $y$-coordinate at time $t$.
And in a bi-causal transport this holds in both directions, i.e. also when exchanging the role of $x$ and $y$.

As already mentioned in the Introduction, the adaptdness or bi-causality constraint turns out to be the correct one to impose on couplings, in order to modify the Wasserstein distance so to ensure robustness of stochastic optimization problems. That is to say, if two measures $\mu,\nu$ are close w.r.t. this distance, then solving w.r.t. $\mu$ optimization problems such as optimal stopping, optimal hedging, utility maximization etc, provides an ``almost optimizer" for $\nu$; see \cite{backhoff2020adapted}.

\begin{definition}[Adapted Wasserstein distance / Nested distance]
    The (first order) \textit{adapted Wasserstein distance} $\AWD$ on $\mathcal{P}(\R^{dT})$ is defined by 
    \begin{equation}\label{eq:adaptedwd}
        \AWD(\mu,\nu) = \inf_{\pi \in \bccpl(\mu,\nu)}\int\sum_{t=1}^{T}\big\lVert x_{t}-y_{t}\big\rVert_{\R^{d}}\pi(dx,dy).
    \end{equation}
\end{definition}
Bi-causal couplings and the corresponding optimal transport problem were considered by R\"{u}schendorf~\cite{ruschendorf1985wasserstein} in so-called `Markov-constructions'. This concept was independently introduced by Pflug-Pichler~\cite{pflug2012distance} in the context of stochastic multistage optimization problems, see also \cite{pflug2014multistage,pflug2015dynamic,pflug2016empirical,glanzer2019incorporating,pichler2013evaluations}, and also considered by  Bion-Nadal and Talay in \cite{bion2019wasserstein} and Gigli in \cite{gigligeometry}. Pflug-Pichler refer to the adapted Wasserstein distance as nested distance, with an alternative representation through dynamic programming principle by disintegrating \eqref{eq:adaptedwd} and replacing conditional laws with \eqref{eq:bicausal}. For notational simplicity, we state it here only for the case $T = 2$, where one obtains the representation 
\begin{equation*}
    \AWD(\mu,\nu) = \inf_{\pi \in \cpl(\mu_1,\nu_1)}\int\Vert x_1 - y_1\Vert_{\R^{d}} + \WD(\mu_{x_1},\nu_{y_1})\pi(dx_1,dy_1).
\end{equation*}
This reflects clearly that $\AWD$ considers not only marginal laws but also the difference between conditional laws. The example below explicitly shows the gap between Wasserstein distance and adapted Wasserstein distance, when conditional laws mismatch. Additionally, when regarding $\mu$ and $\nu$ as distributions of risky assets, it clearly illustrates the inappropriateness of the Wasserstein distance to gauge closeness of financial markets, and the way in which its adapted counterpart amends to it. 
\begin{example}
    Let $\mu, \nu \in \mathcal{P}([0,1]^2)$ be given by $\mu = \frac{1}{2}\delta_{(0,1)} + \frac{1}{2}\delta_{(0,-1)}$ and $\nu = \frac{1}{2}\delta_{(\epsilon,1)} + \frac{1}{2}\delta_{(-\epsilon,-1)}$, with $\epsilon \in (0, 1)$, visualized in Figure \ref{fig:test2}. Then $\WD(\mu,\nu) = \epsilon$, while $\AWD(\mu,\nu) = 1 + \epsilon$, since $\WD(\mu_{x_1},\nu_{y_1}) = 1$ no matter how we couple the first coordinate.
    \begin{figure}[ht]
        \centering
        \begin{subfigure}{.4\textwidth}
          \centering
          \includegraphics[width=0.6\linewidth]{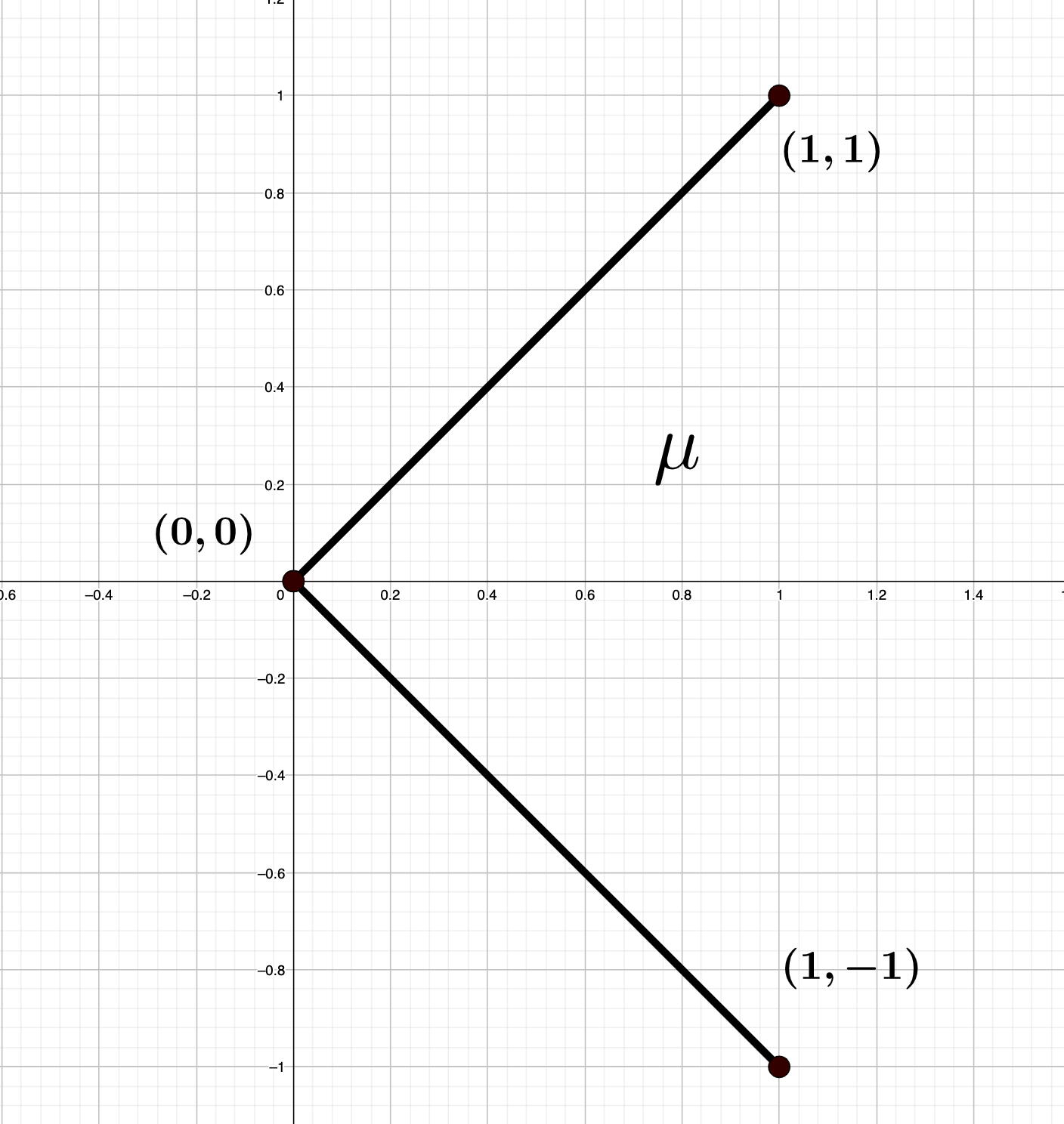}
          \caption{$\mu = \frac{1}{2}\delta_{(0,1)} + \frac{1}{2}\delta_{(0,-1)}$}
          \label{fig:sub11}
        \end{subfigure}
        \begin{subfigure}{.4\textwidth}
          \centering
          \includegraphics[width=0.6\linewidth]{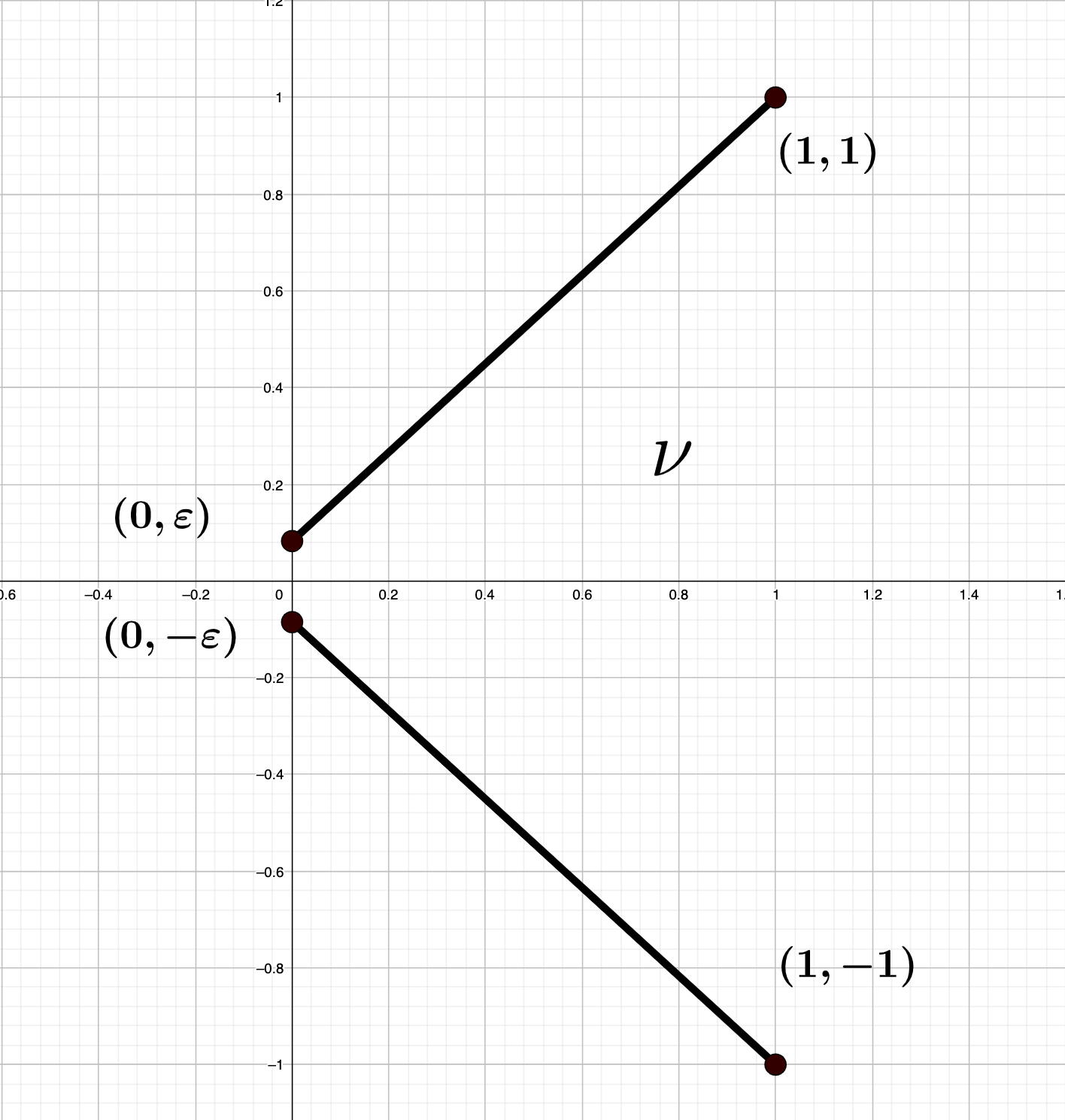}
          \caption{$\nu = \frac{1}{2}\delta_{(\epsilon,1)} + \frac{1}{2}\delta_{(-\epsilon,-1)}$}
          \label{fig:sub12}
        \end{subfigure}
        \caption{Visualization of $\mu$ and $\nu$.}
        \label{fig:test2}
    \end{figure}
\end{example}    
Let us now consider a financial market with an asset whose law is described by $\mu$, and another market with an asset whose law is described by $\nu$. Then under the Wasserstein distance the two markets are judged as being close to each other, while they clearly present very different features (random versus deterministic evolution, no-arbitrage versus arbitrage, etc.). It is also evident how optimization problems in the two situations would lead to very different decision making. This is a standard example to motivate the introduction of adapted distances, that instead can distinguish between the two models.

\subsection{Adapted empirical measure}
Let $\mu \in \mathcal{P}(\R^{dT})$, and let $(X^{(n)})_{n\in\N}$ be i.i.d. samples from $\mu$ defined on some probability space $(\Omega, \mathcal{F},\P)$. 
\begin{definition}[Empirical measure]
    For all $N\in\N$, we denote  by
    \[
    \mu^{N} = \frac{1}{N}\sum_{n=1}^{N}\delta_{X^{(n)}}
    \]
    the \textit{empirical measure} of $\mu$.
\end{definition}
Next we introduce the adapted empirical measure on $\R^{dT}$ analogously to the adapted empirical measure defined on $[0,1]^{dT}$ in \cite{backhoff2022estimating}.
\begin{definition}[(Uniform) adapted empirical measure]\label{def:uniform_adapted_empirical_measure}
    For $N \in \N$ and grid size $\Delta_N > 0$, we let $m = \lceil \frac{1}{\Delta_N}\rceil$ and consider the uniform partition $\hat{\Phi}$ of $\R^{dT}$ given by
    \begin{equation}
    \label{eq:def:Phi1}
        \hat{\Phi}^{N} = \left\{\hat{\mathcal{C}}_{\mathbf{z}}^{N} = \Big[0,\frac{1}{m}\Big]^{dT} + \frac{1}{m} \mathbf{z} , \mathbf{z} \in \mathbb{Z}^{dT}\right\}.
    \end{equation}
    Let $\hat{\Lambda}^N$ be the set of mid points of all cubes $\hat{\mathcal{C}}_{\mathbf{z}}^{N}$ in the partition $\hat{\Phi}^{N}$, and let $\hat{\varphi}^{N}\colon \R^{dT} \to \hat{\Lambda}^N$ map each cube $\hat{\mathcal{C}}_{\mathbf{z}}^{N}$ to its mid point \add{(points belonging to more than one cube can be mapped into any of them).}
    Then we denote by
    \[
    \hat{\mu}^{N} = \frac{1}{N}\sum_{n=1}^{N}\delta_{\hat{\varphi}^{N}(X^{(n)})}
    \]
    the \textit{(uniform) adapted empirical measure} of $\mu$ with grid size $\Delta_N$.
\end{definition}
Intuitively, the adapted empirical measure is constructed via the following procedure: we tile $\R^{dT}$ with cubes of size $(\frac{1}{m})^{dT}$ that form the partition $\hat{\Phi}^{N}$, then we project all points in each cube $\hat{\mathcal{C}}_{\mathbf{z}}^{N}$ to its mid point, see Figure~\ref{fig:sub1}. As a result, the push-forward measure obtained as empirical measure of the samples after projections is indeed the adapted empirical measure. 

On the other hand, if we tile $\R^{dT}$ with cubes of different sizes, then we obtain a non-uniform partition, on which we can similarly consider a projection and from it define a non-uniform adapted empirical measure. For the purpose of this paper, we consider the partition which is denser around the origin and coarser when far from it.
\
\begin{definition}[Non-uniform adapted empirical measure]\label{def:nonuniform_adapted_empirical_measure}
    For $N \in \N$ and grid size $\Delta_N > 0$, we let $m = \lceil \frac{1}{\Delta_N}\rceil$, $\mathcal{A}_0 = [-1,1]^{dT}$. Let $\mathcal{A}_{j} = [-2^j,2^{j}]^{dT} \add{\backslash} [-2^{j-1},2^{j-1}]^{dT}, j\in\N,$ be cubic rings and
    \begin{equation}
    \label{eq:def:Phi2}
        \check{\Phi}^{N} = \bigcup_{j=0}^{\infty}\check{\Phi}^{N,j} = \bigcup_{j=0}^{\infty}\left\{\check{\mathcal{C}}_{\mathbf{z}}^{N,j} = \Big[0,\frac{2^{j-1}}{m}\Big]^{dT} + \frac{2^{j-1}}{m} \mathbf{z} \subseteq \mathcal{A}_{j} , \mathbf{z} \in \mathbb{Z}^{dT}\right\}.
    \end{equation}
    Let $\check{\Lambda}^N$ be the set of mid points of all cubes $\check{\mathcal{C}}_{\mathbf{z}}^{N,j}$ in the partition $\check{\Phi}^{N}$, and let $\check{\varphi}^{N}\colon \R^{dT} \to \check{\Lambda}^N$ map each cube $\check{\mathcal{C}}_{\mathbf{z}}^{N,j}$ to its mid point. Then we denote by 
    \[
    \check{\mu}^{N} = \frac{1}{N}\sum_{n=1}^{N}\delta_{\check{\varphi}^{N}(X^{(n)})}
    \]
    the \textit{non-uniform adapted empirical measure} of $\mu$ with grid size $\Delta_N$.
\end{definition}
Intuitively, we first partition $\R^{dT}$ into cubic rings $\mathcal{A}_{j}$, $j=0,1,2,...$, and then partition each ring $\mathcal{A}_{j}$ into subcubes of size $(\frac{2^{j-1}}{m})^{dT}$, see Figure~\ref{fig:sub2}. By doing this, we avoid considering a
fine partition when we are far from the origin where less mass of the measure lies. 

Throughout the paper, we will use the following notation. We use the symbol $\wedge$ for notations related to the (uniform) adapted empirical measure, such as $\hat{\mu}^{N}, \hat{\Phi}^{N}, \hat{\Phi}^{N}_{t}, \hat{\varphi}^{N}, \hat{\Lambda}^{N}$, and the symbol $\vee$ for notations related to the non-uniform adapted empirical measure, such as $\check{\mu}^{N}, \check{\Phi}^{N}, \check{\Phi}^{N}_{t}, \check{\varphi}^{N}, \check{\Lambda}^{N}$. If a statement holds for both uniform and non-uniform cases, we instead adopt the notations $\bar{\mu}^{N},\Phi^{N}, \Phi^{N}_{t},\Lambda^{N}_{t},\varphi^{N}_{t}$, \add{where we use the bar symbol} on the measure to avoid duplication of the notation $\mu^{N}$, which we use for the classical empirical measure. \add{We let $\Delta_N =N^{ -\frac{1}{dT}}$, for all $N\in\N$.}
\begin{figure}[ht]
    \centering
    \begin{subfigure}{.4\textwidth}
      \centering
      \includegraphics[width=0.6\linewidth]{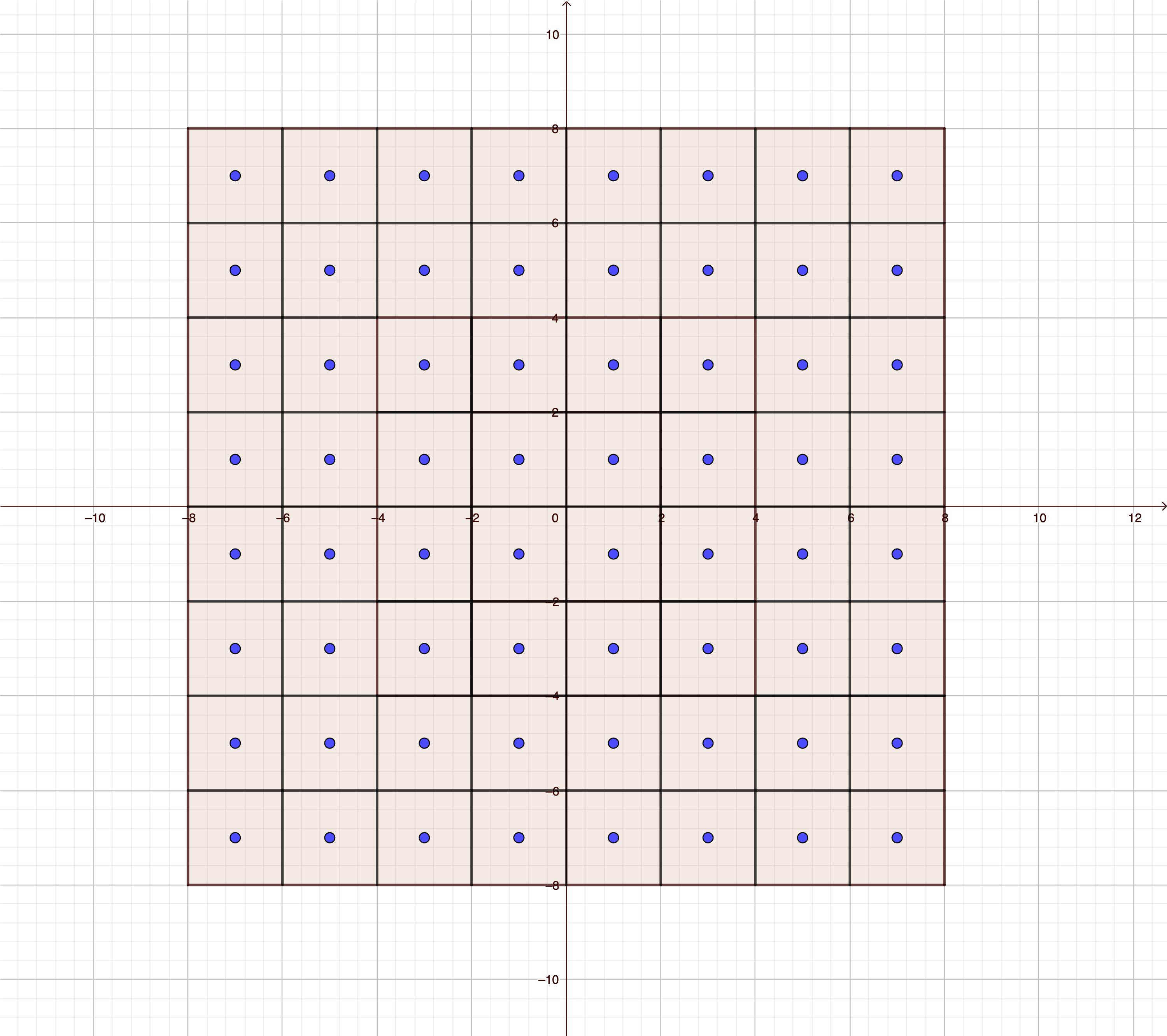}
      \caption{Uniform grid}
      \label{fig:sub1}
    \end{subfigure}
    \begin{subfigure}{.4\textwidth}
      \centering
      \includegraphics[width=0.6\linewidth]{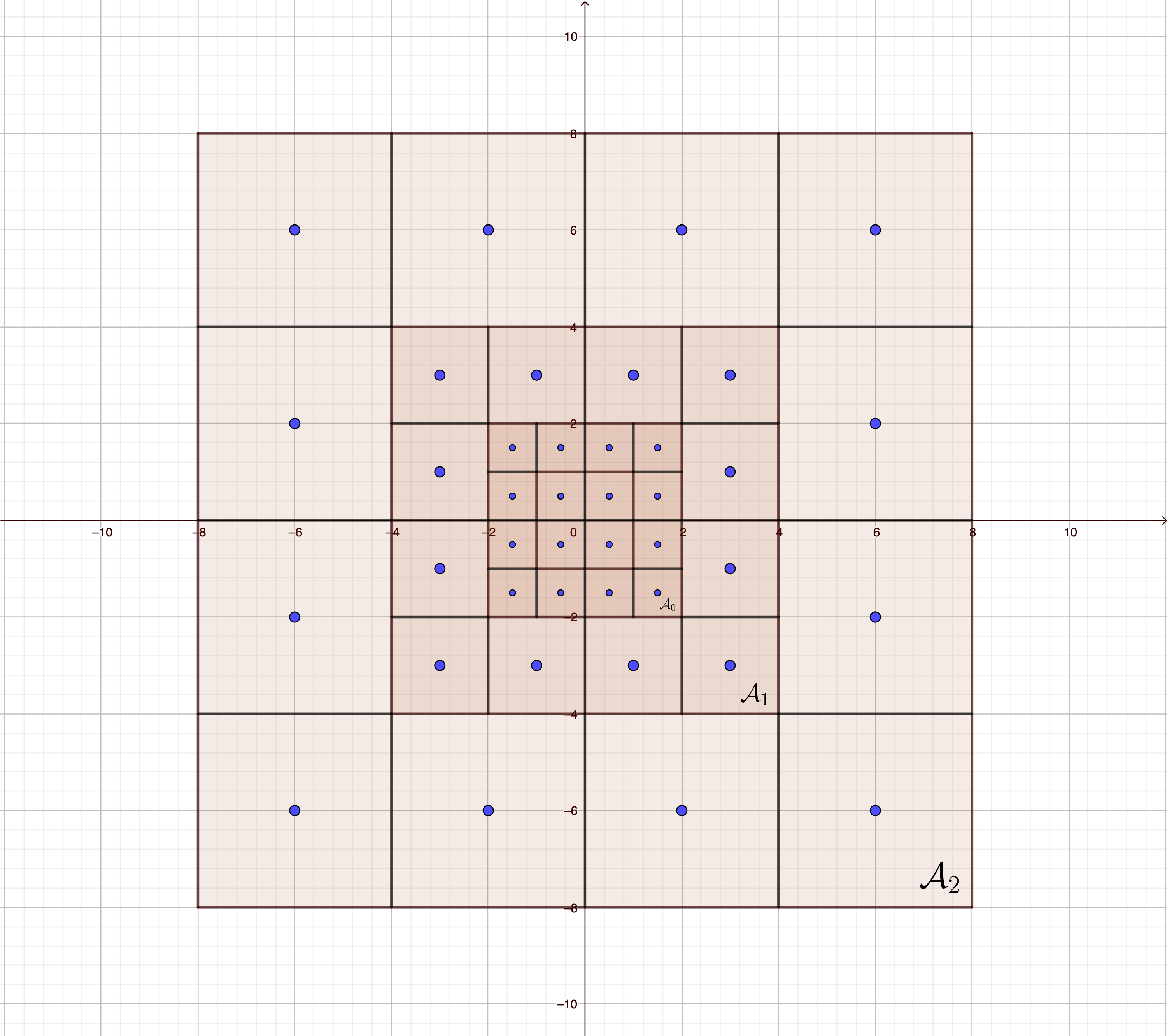}
      \caption{Non-uniform gird}
      \label{fig:sub2}
    \end{subfigure}
    \caption{Visualization of uniform grid and non-uniform grid.}
    \label{fig:test}
\end{figure}
\subsection{Main results}
Our first main theorem is the following consistency result: the adapted empirical measure converges almost surely to the underlying measure $\mu$ under adapted Wasserstein distance, as long as $\mu$ is integrable; see Section~\ref{sec:asconvergence} for the proof. 
\begin{theorem}[Almost sure convergence]
\label{thm:asconvergence}
    Let $\mu \in \mathcal{P}(\R^{dT})$ be integrable. Then 
    \begin{equation*}
        \lim_{N \to \infty} \AWD(\mu,\hat{\mu}^{N}) = 0\quad \P\text{-a.s.}
    \end{equation*}
    and
    \begin{equation*}
        \lim_{N \to \infty} \AWD(\mu,\check{\mu}^{N}) = 0\quad \P\text{-a.s.}
    \end{equation*}
\end{theorem}
It is worth noticing that integrability is the mildest assumption here because at least we need $\AWD(\mu,\mu^{N})$ to be finite. In order to further quantify the speed of convergence for $\mu \in \calP(\R^{dT})$, integrability is not enough anymore. Then we need two more regularity assumptions: the uniform moment noise condition and the uniform Lipschitz kernel condition. 

For $p\geq 1$, we denote by $M_{p}(\mu) = \int \Vert x\Vert^{p}\mu(dx)$ the $p$-th moment of $\mu$, and for $\alpha,\gamma > 0$, we denote by $\mathcal{E}_{\alpha,\gamma}(\mu) = \int \exp\Big(\gamma\Vert x\Vert^{\alpha}\Big)\mu(dx)$ the $(\alpha,\gamma)$-exponential moment of $\mu$. Besides the moment condition in \cite{fournier2015rate} on the underlying measure, we would also assume some moment condition on kernels. The simplistic way to require this, would be to assume uniform moment of kernels i.e. $\sup_{x_{1:t} \in \R^{dt}}M_{p}(\mu_{x_{1:t}}) < \infty$
for all $t=1,\dots,T-1$. However, this is such a strong condition that even the discretized Black-Scholes model (with simplest dynamic: $X_{t+1} = X_t + X_t \Delta W_t$) fails to satisfy! Motivated by this, we define the uniform moment noise assumption based on kernel decomposition. We start by noticing that, for every $\mu\in\mathcal{P}(\R^{dT})$ and for all $t=1,\dots,T-1$ and $x_{1:t}\in\R^{d}$, the kernel $\mu_{x_{1:t}}$ admits the following decomposition: 
\[
    \mu_{x_{1:t}} \sim f_t(x_{1:t}) + \sigma_{t}(x_{1:t})\cdot \varepsilon_{x_{1:t}},
\]
where $f_{t}\colon \R^{dt} \to \R^{d}$, $\sigma_t\colon \R^{dt}\to\R_{\geq 0}$, and $\varepsilon_{x_{1:t}}$ is an $\R^d$-valued random variable depending on $x_{1:t}$. Indeed, to see this it suffices to consider the following decomposition of $(X_t)_{t=1}^{T} \sim \mu$, for any $p\geq 1$ s.t. $M_{p}(\mu) < \infty$:
\begin{equation}  \label{eq:kernel_decomposition}
    X_{t+1} = \E[X_{t+1}\vert\mathcal{F}^X_t] + \E\Big[\big|X_{t+1} - \E[X_{t+1}\vert\mathcal{F}^X_t]\big|^{p} \Big\vert \mathcal{F}^X_t\Big]^{1/p}  \frac{X_{t+1} - \E[X_{t+1}\vert\mathcal{F}^X_t]}{\E\Big[\big|X_{t+1} - \E[X_{t+1}\vert\mathcal{F}^X_t]\big|^{p} \Big\vert \mathcal{F}^X_t \Big]^{1/p}},
\end{equation}
where ${(\mathcal{F}^X_t)}_t$ is the filtration generated by $(X_t)_t$.
\begin{definition}[Uniform moment \add{noise} kernel]\label{def:uniform_kernels}
    Let $\mu\in\mathcal{P}(\R^{dT})$, $r \geq 0$, $p \geq 1$ (resp. $\alpha,\gamma > 0$). We say that $\mu$ has uniform $p$-th moment (resp. uniform $(\alpha,\gamma)$-exponential moment) \add{noise} kernel of growth $r$ if there exist $c>0$, $f_{t}\colon \R^{dt} \to \R^{d}$, $\sigma_t\colon \R^{dt}\to\R_{\geq 0}$ and $\lambda_{x_{1:t}} \in \calP(\R^{d})$ s.t., for all $t = 1,\dots,T-1$, $x_{1:t}\in\R^{dt}$,
    \begin{equation}
    \label{eq:def:uniform_kernels}
        \mu_{x_{1:t}} \sim f_t(x_{1:t}) + \sigma_{t}(x_{1:t})\cdot \varepsilon_{x_{1:t}},\quad  \varepsilon_{x_{1:t}}\sim \lambda_{x_{1:t}},
    \end{equation}
    with $\sigma_{t}(x_{1:t}) \leq c\Vert x_{1:t}\Vert^r$ and $\sup_{x_{1:t} \in \R^{dt}}M_{p}(\lambda_{x_{1:t}}) < \infty$ (resp. $\sup_{x_{1:t} \in \R^{dt}}\mathcal{E}_{\alpha,\gamma}(\lambda_{x_{1:t}}) < \infty$).
\end{definition}
{\color{black}
\begin{remark}
    If we consider $f_t=0$ and $\sigma_t=1$ for all $t=1,\dots,T-1$, then the uniform moment noise kernel condition boils down to the uniform moment kernel condition, i.e. $\sup_{x_{1:t} \in \R^{dt}}M_{p}(\mu_{x_{1:t}}) < \infty$ for all $t=1,\dots,T-1$. In this case, one can directly follow the proof in \cite{backhoff2022estimating} without involving the shifting and scaling kernels technique, but still using different estimations than those in \cite{backhoff2022estimating}, in order to deal with the infinite sum of error terms when the support is not compact.
\end{remark}
}
Moreover, by the kernel decomposition \eqref{eq:kernel_decomposition}, we can directly see from Definition~\ref{def:uniform_kernels} that as long as $(X_t)_{t=1}^{T} \sim \mu$ satisfies 
\begin{equation*}
    \sup_{x_{1:t}\in\R^{dt}}M_{p}\left(\frac{X_{t+1} - \E[X_{t+1}|X_{1:t} = x_{1:t}]}{\Vert x_{1:t} \Vert^{r} }\right) < \infty,
\end{equation*}
then $\mu$ has uniform $p$-th moment noise kernels of growth $r$. This gives a general way to test whether a distribution has uniform moment noise kernel. Next, we introduce the uniform Lipschitz kernels in the same fashion of \cite{pflug2016empirical, backhoff2022estimating}.
\begin{definition}
    We say that $\mu\in\mathcal{P}(\R^{dT})$ has $L$-Lipschitz kernels if there exists a disintegration of $\mu$ such that, for all $t=1,\dots,T-1$,
    \begin{equation*}
        \R^{dt} \ni (x_1,\dots,x_t) \mapsto \mu_{x_{1:t}} \in \mathcal{P}(\R^{d})
    \end{equation*}
    is $L$-Lipschitz continuous on $\R^{dt}$, where $\mathcal{P}(\R^{d})$ is endowed with Wasserstein distance.
\end{definition}
We show below two examples where both the uniform moment \add{noise} and the uniform Lipschitz kernel conditions hold.
\begin{example}[Discretized Black Scholes model]
    Let $\mu\in \calP(\R^{dT})$ be the law of the discretized Black Scholes model, satisfying $\mu_1 = \delta_{\mathbf{1}}$ and
    \[
        \mu_{x_{1:t}} \sim \add{x_t} + \sqrt{\Delta t}\,\add{x_t}\cdot \varepsilon_{t},
    \]  
    where $\varepsilon_{t} \sim \calN(0,\mathbf{I}_d)$ i.i.d. for all $t = 1,\dots, T-1$. Notice that
    \begin{equation*}
        \begin{split}
            \WD(\mu_{x_{1:t}}, \mu_{x'_{1:t}}) \leq \WD_2(\mu_{x_{1:t}}, \mu_{x'_{1:t}}) &\leq \sqrt{\Vert \add{x_t} - \add{x'_t}\Vert^2 + \Vert \sqrt{\Delta t}\add{x_t} - \sqrt{\Delta t}\add{x'_t} \Vert^2 \mathrm{Trace}(\mathbf{I}_d)}\\
            &\leq (1+d\Delta t)^{\frac{1}{2}}\Vert \add{x_t} - \add{x'_t}\Vert \leq \add{(1+d\Delta t)^{\frac{1}{2}}\Vert x_{1:t} - x'_{1:t}\Vert}.
        \end{split}
    \end{equation*}
    Then, for all $p,q > 0$, $\mu$ has finite $q$-th moment and $(1+dt)^{\frac{1}{2}}$-Lipschitz, uniform $p$-th moment kernels of growth $1$.
\end{example}
{\color{black}
    \begin{example}[Discretized SDE]
    \label{ex:sde}
    Let us consider the following SDE:
    \begin{equation*}
        dX_t  = \mu(t,X_t)dt + \sigma(t,X_t) dW_t,
    \end{equation*}
    where $\mu$ and $\sigma$ are Lipschitz w.r.t. the second argument. The corresponding discretized SDE then reads as
    \begin{equation*}
        X_{t+1} = X_t + \mu(t,X_t)\Delta t + \sigma(t,X_t) \sqrt{\Delta t} \epsilon_t,
    \end{equation*}
    where $(\epsilon_t)_t$ are independent standard normal random variables. Notice that, as $\mu$ and $\sigma$ are Lipschitz, they have at most linear growth. We conclude that the solution of the above discretized SDE has uniform exponential moment noise of growth $1$ and Lipschitz kernels. 
    \end{example}
    \begin{remark}
        Many other benchmark models in mathematical finance satisfy uniform moment noise kernels, e.g. the GARCH model and the Heston model with clipped volatility, which could be checked similarly as in Example~\ref{ex:sde}. 
    \end{remark}}
Recall that under the classical Wasserstein distance, the rates are as follows: 
\begin{equation*}
\E[\WD_1(\mu,\mu^N)] \leq \left\{\begin{aligned}
    &CN^{-\frac{1}{2}},\quad  &d=1\\ 
    &CN^{-\frac{1}{2}}\log(N) ,\quad &d=2\\
    &CN^{-\frac{1}{d}} ,\quad  &d\geq 3
    \end{aligned}\right\} \leq  CN^{-\frac{1}{\calD(d)}},
\end{equation*}
where  $\calD\colon\N \to \N$ is defined as
\begin{equation*}
\calD(d) = \left\{\begin{aligned}
    &d+1, &d=1,2\\ 
    &d, &d\geq 3
    \end{aligned}\right. .
\end{equation*}
Now, we are ready to state our main speed of convergence result. Since moment conditions for the same convergence rate are slightly different between the uniform and non-uniform adapted empirical measures, we consider the following two settings. 
\begin{setting}\label{setting:1}
    Let $r\geq 0$, $p \geq 1$,   $q > \max\left\{\frac{\add{\calD(d)}}{\add{\calD(d)}-1}(r + T -1),rp + \add{\calD(d)}(T-1)(p-1)\right\}$, and let $\mu \in \mathcal{P}(\R^{dT})$ with finite $q$-th moment and $L$-Lipschitz, uniform $p$-th moment kernels of growth $r$. \add{Moreover, let $L_f >0$ and assume that $f_t$ in \eqref{eq:def:uniform_kernels} is $L_f$-Lipschitz for all $t=1,\dots,T-1$.}
\end{setting}
\begin{setting}\label{setting:2}
    Let $r\geq 0$, $p \geq 1$,  $q > \max\left\{\frac{r\add{\calD(d)}}{\add{\calD(d)}-1},(r+1)p\right\}$, and let $\mu \in \mathcal{P}(\R^{dT})$ with finite $q$-th moment and $L$-Lipschitz, uniform $p$-th moment kernels of growth $r$. \add{Moreover, let $L_f >0$ and assume that $f_t$ in \eqref{eq:def:uniform_kernels} is $L_f$-Lipschitz for all $t=1,\dots,T-1$.}
\end{setting}
First we estimate the speed of the expected error; see Section~\ref{sect:mom_est} for the proof.

\begin{theorem}[Moment estimate]
    \label{thm:moment_estimate}
    Let $p,q > 1$.
    \begin{enumerate}[(i)]
        \item Under Setting \ref{setting:1}, there exists a constant $C>0$ s.t., for all $N\in\N$,
    \begin{equation*}
        \E\big[\AWD(\mu,\hat{\mu}^{N})\big] \leq CN^{-\frac{1}{\add{\calD(d)}T}} + CN^{-\frac{p-1}{pT}}.
    \end{equation*}
        \item Under Setting \ref{setting:2},  there exists a constant $C>0$ s.t., for all $N\in\N$,
    \begin{equation*}
        \E\big[\AWD(\mu,\check{\mu}^{N})\big] \leq CN^{-\frac{1}{\add{\calD(d)}T}} + CN^{-\frac{p-1}{pT}}.
    \end{equation*}
    \end{enumerate}
\end{theorem}  
When $p > 1 + \frac{1}{\add{\calD(d)}-1}$, $N^{-\frac{1}{\add{\calD(d)}T}}$ dominates the right hand side in the two above estimates. Next, we estimate the speed of deviation error; see Section~\ref{sec:concentration1} for the proof.
\begin{theorem}[Concentration inequality I]
    \label{thm:first_concentration}
    Let $p,q > 2$ and $\mathrm{rate}(N) = \mathcal{O}(N^{-\frac{1}{\add{\calD(d)}T}})$.
    \begin{enumerate}[(i)]
        \item Assume Setting \ref{setting:1}. Then, for all $\epsilon \in (0,p-2)$, there exists a constant $C > 0$ s.t., for all $N \in \N$ and $x > 0$,
        \begin{equation*}
        \P\Big[\AWD(\mu,\hat{\mu}^{N}) \geq x + \mathrm{rate}(N)\Big] \leq \add{CN^{-\frac{p-\epsilon}{\add{\calD(d)}T}}x^{-p+\epsilon}}.
        \end{equation*}
        \item Assume Setting \ref{setting:2}. Then, for all $\epsilon \in (0,p-2)$, there exists a constant $C > 0$ s.t., for all $N \in \N$ and $x > 0$,
        \begin{equation*}
        \P\Big[\AWD(\mu,\check{\mu}^{N}) \geq x + \mathrm{rate}(N)\Big] \leq \add{CN^{-\frac{p-\epsilon}{\add{\calD(d)}T}}x^{-p+\epsilon}} + C N^{-\frac{p}{2}} x^{-p}.
        \end{equation*}
    \end{enumerate}
\end{theorem}
Moreover, if we strengthen the moment condition to the exponential moment condition, we obtain the (optimal) exponential concentration inequality; see Section~\ref{sec:concentration2} for the proof.
\begin{setting}\label{setting:3}
    Let $\alpha \geq 2$, $\gamma > 0$, \add{$r \geq 0$}, and $\mu \in \mathcal{P}(\R^{dT})$ with finite $(\alpha,\gamma)$-exponential moment and $L$-Lipschitz, uniform $(\alpha,\gamma)$-exponential moment kernels of growth $r$.
\end{setting}
\begin{theorem}[Concentration inequality II]
    \label{thm:second_concentration}
    Assume Setting \ref{setting:3} and let $\mathrm{rate}(N) = \mathcal{O}(N^{-\frac{1}{\add{\calD(d)}T}})$.
    \begin{enumerate}[(i)]
        \item For uniform adapted empirical measures and $r=0$, there exist constants $c,C > 0$ such that, for all $x \geq 0$ and $N \in \N$,
        \begin{equation*}
            \P\Big[\AWD(\mu,\hat{\mu}^{N}) \geq  x + \mathrm{rate}(N) \Big] \leq Ce^{-cNx^{2}}.
        \end{equation*}
        \item For uniform adapted empirical measures and $r>0$, there exist constants $c,C > 0$ such that, for all $x \geq 0$ and $N \in \N$ satisfying $\sqrt{N}x\geq1$,
        \begin{equation*}
            \P\Big[\AWD(\mu,\hat{\mu}^{N}) \geq  x + \mathrm{rate}(N) \Big] \leq CNe^{-c (\sqrt{N}x)^{\frac{1}{r+1}}}.
        \end{equation*}
        \item For non-uniform adapted empirical measures, there exist constants $c,C > 0$ such that, for all $x \geq 0$ and $N \in \N$ satisfying $\sqrt{N}x\geq1$,
        \begin{equation*}
            \P\Big[\AWD(\mu,\check{\mu}^{N}) \geq  x + \mathrm{rate}(N) \Big] \leq CNe^{-c (\sqrt{N}x)^{\frac{1}{r+2}}}.
        \end{equation*}
    \end{enumerate}
\end{theorem}
Theorem~\ref{thm:second_concentration} (i) recovers the same optimal exponential concentration rate of compactly supported adapted empirical measures under adapted Wasserstein distance, that is Theorem 1.7 in \cite{backhoff2022estimating}. Moreover, this is also the same exponential concentration rate of empirical measures under Wasserstein distance under exponential moment assumption, that is Theorem~2 in \cite{fournier2015rate}. 

For the convenience of the reader, we 
recall the convergence rates of adapted empirical measures given in \cite{backhoff2022estimating}.
\begin{theorem}[Theorem 1.3 in \cite{backhoff2022estimating}]
    \label{thm:asconvergence_cube}
    Let $\mu \in \mathcal{P}([0,1]^{dT})$. Then $\lim_{N \to \infty} \AWD(\mu,\hat{\mu}^{N}) = 0\quad \P\text{-a.s.}$.
    \end{theorem}
\begin{theorem}[Theorem 1.5 and Theorem 1.7 in \cite{backhoff2022estimating}]\label{thm:errorestimation_cube}
    Assume $\mu \in \mathcal{P}([0,1]^{dT})$ with $L$-Lipschitz kernels. Then there exists $C > 0$ s.t., for all $N \in \N$ and $x > 0$,
    \begin{equation*}
        \E[\AWD(\mu,\hat{\mu}^{N})] \leq CN^{-\frac{1}{\calD(d)T}}
    \end{equation*}
    and 
    \begin{equation*}
    \P\Big[\AWD(\mu,\hat{\mu}^{N}) \geq \mathrm{rate}(N) + x\Big] \leq Ce^{-cNx^2},
    \end{equation*}
where $\mathrm{rate}(N) = \mathcal{O}(\Delta_N)$.
\end{theorem}

\begin{remark}
If $\mu \in \calP([0,1]^{dT})$, then the integrability assumptions in Theorem~\ref{thm:asconvergence}, Theorem~\ref{thm:moment_estimate} and Theorem~\ref{thm:second_concentration} hold for all $p > 1$. Therefore, our results generalize those in \cite{backhoff2022estimating}. In particular we recover the sharp rates when $d=1$ and $d\geq 3$.

    When $d=2$, the sharp moment rate is of growth $N^{-\frac{1}{d}}\log(N)$. However, the $\log(N)$ term is estimated by $N^{\frac{1}{d} - \frac{1}{d+\zeta}}$ for all $\zeta >0$ in our results, in particular $\zeta = 1$ in the statements of main theorems. Thus, we achieve at most a $N^{\frac{1}{d+\zeta}}$ rate for all $\zeta >0$, and not the optimal one $N^{-\frac{1}{d}}\log(N)$. 
\end{remark}

\section{Notations and preparatory lemmas}\label{sec:preparation}
We start this section by recalling some notations introduced above, and introducing some more, that will be used throughout the paper.
For $t = 1,\dots,T$, we denote by $\Vert \cdot \Vert_{1:t} \colon \R^{dt} \to \R$ the sum-norm on $\R^{dt}$ defined by $\Vert x_{1:t} \Vert_{1:t} = \sum_{s = 1}^{t} \Vert x_{s} \Vert_{\R^{d}}$ for all $x_{1:t} = (x_1,\dots,x_t) \in \R^{dt}$. For simplicity, we unify notations of norm and write $\Vert x_{1:t} \Vert = \Vert x_{1:t} \Vert_{1:t}$ for all $x_{1:t} \in \R^{dt}$ and $t = 1,\dots,T$. For $\mu \in \mathcal{P}(\R^{dT})$ and $t=1,\dots,T$, we denote by $\mu_{t}\in \mathcal{P}(\R^{d})$ the $t$-th marginal of $\mu$, and by $\mu_{1:t}\in \mathcal{P}(\R^{dt})$ the up to $t$ marginal of $\mu$, in particular $\mu_{1}$ is the first marginal. For all $\mu \in \mathcal{P}(\R^{dT})$ and Borel sets $G \subseteq \R^{dt}$, $t=1,\dots,T$, we also unify notations of marginals and denote $\mu(G) = \mu_{1:t}(G)$. For all $x_{1:t} \in \R^{dt}$, $t=1,\dots,T-1$, we denote by $\mu_{x_{1:t}} \in \mathcal{P}(\R^{d})$ the kernel of $\mu$. Moreover, for all Borel sets $G \subseteq \R^{dt}$, $t=1,\dots,T-1$, we define the conditional law of $\mu$ on $G$ by $\mu\vert_G\in \mathcal{P}(\R^{dt})$:
\begin{equation*}
\mu\vert_G(\add{dx_{1:t}}) = \frac{\mu(G\cap dx_{1:t}  )}{\mu(G)},
\end{equation*}
and define the average kernel of $\mu$ on $G$ by $\mu_{G} \in \mathcal{P}(\R^{d})$:
\begin{equation*}
    \mu_G(dx_{t+1}) = \frac{1}{\mu(G)}\int_{G} \mu_{x_{1:t}}(dx_{t+1}) \mu(dx_{1:t}).
\end{equation*}
If $\mu(G) = 0$, w.l.o.g. we set $\mu\vert_G = \mu_G = \delta_0$ for the completeness of definition. For $G \subseteq \R^{dT}$, we define $\mathrm{diam}(G) = \sup_{x,y\in G}\Vert x - y \Vert$ and refer to it as the diameter of $G$, and let $\Vert G\Vert = \sup_{x\in G}\Vert x \Vert$. For example, for any cube $G = [0,\ell]^{dt}$, $\ell >0$, we have $\mathrm{diam}(G) = \Vert G \Vert = t\sqrt{d}\ell^{d}$. \add{For $f \in \R^{d}$, $\sigma > 0$ and $\gamma \in \calP(\R^{d})$, we use the notation $\gamma + f = (x\mapsto x+f)_{\#} \gamma$ and $
    \sigma \gamma = (x\mapsto \sigma x)_{\#} \gamma$.}
Recall the notations introduced in Definitions~\ref{def:uniform_adapted_empirical_measure} and  \ref{def:nonuniform_adapted_empirical_measure} on $\R^{dT}$, which we can analogously define on $\R^{dt}$ \add{by replacing the time dimension $T$ with $t$, e.g. $\hat{\Phi}^{N}_t = \big\{\big[0,\frac{1}{m}\big]^{dt} + \frac{1}{m} \mathbf{z} , \mathbf{z} \in \mathbb{Z}^{dt}\big\}$ and
$\mathcal{A}_{j}^t = [-2^j,2^{j}]^{dt}\backslash[-2^{j-1},2^{j-1}]^{dt}$.} Similarly, replacing $T$ with $t$, we define $\check{\Phi}^{N}_{t}$, $\Phi^{N}_{t}$, $\hat{\Phi}^{N,j}_t$, $\check{\Phi}^{N,j}_t$, $\Phi^{N,j}_t$, $\hat{\Lambda}^{N}_t$, $\check{\Lambda}^{N}_t$, $\Lambda^{N}_t$, $\hat{\varphi}^{N}_{t}$, $\check{\varphi}^{N}_{t}$, $\varphi^{N}_{t}$.\\[0.1cm]
\emph{Convention:} In our proofs, we let $C,c > 0$ be generic constants that may increase from line to line  depending on all sorts of external parameters, e.g. $\epsilon$ and $L$, but independent of internal parameters, e.g. $N\in\N$ and $x > 0$. Moreover, to simplify notations, we present the proof for the case $d \geq 3$ (so that $\calD(d)=d$), while for the general case one would need to replace $d$ by $\calD(d)$ in the proof.

\subsection{Preparatory lemmas}
The uniform  and non-uniform adapted partitions are different in the diameter of the cubes that form the partition, i.e. $\mathrm{diam}(G)$ for $G \in \Phi_t^N$.
Thus, they are also different in the number of cubes in the same cubic ring, which we denote by $\vert \Phi^{N,j}_t \vert$. 
In fact, we have the following estimates of $\mathrm{diam}(G)$ and $\vert \Phi^{N,j}_t \vert$:
\begin{itemize}
    \item For all $t = 1,\dots,T$, $j\geq 0$, $\hat{G} \in \hat{\Phi}^{N,j}_t$,
    \begin{equation}\label{eq:estimate_helper:1}
        \mathrm{diam}(\hat{G}) \leq t\sqrt{d}\Delta_N, \quad \vert \hat{\Phi}^{N,j}_t \vert \leq (2^{j+1}N^{\frac{1}{dT}})^{dt}.
    \end{equation}
    \item For all $t = 1,\dots,T$, $j\geq 0$, $\check{G} \in \check{\Phi}^{N,j}_t$, 
    \begin{equation}\label{eq:estimate_helper:2}
        \mathrm{diam}(\check{G}) \leq t\sqrt{d}2^{j}\Delta_N, \quad \vert \check{\Phi}^{N,j}_t \vert \leq (4N^{\frac{1}{dT}})^{dt}.
    \end{equation}
\end{itemize}
Moreover, since $\Phi^{N,j}_{t}$ is a partition of $\mathcal{A}^{j}_{t}$, we also need to estimate $\mu(\mathcal{A}^{j}_{t})$ in later proofs, which is a direct consequence of Markov's inequality:
\begin{equation*}
    \mu(\mathcal{A}^{t}_{0}) \leq 1,\quad \mu(\mathcal{A}^{t}_{j})\leq \frac{M_{q}(\mu_{1:t})}{2^{q(j-1)}},~ j\in\N.
\end{equation*}
Therefore, if we assume Setting \ref{setting:1} or Setting \ref{setting:2}, then there exists $C > 0$ such that 
\begin{equation}\label{eq:markov2}
    \mu(\mathcal{A}^{t}_{j})\leq \frac{C}{2^{q(j-1)}},~ j\geq 0.
\end{equation}
Next, we present a lemma decomposing the adapted Wasserstein distance between (uniform or non-uniform) adapted empirical measure and the underlying measure, under uniform Lipschitz kernel assumption. In this lemma, the bound for the uniform adapted empirical measure is sharper than for the non-uniform one. 
\begin{lemma}\label{lem:helper}
    Assume $\mu \in \mathcal{P}(\R^{dT})$  has $L$-Lipschitz kernels. Then
    \begin{enumerate}[(i)]
        \item there exists $C > 0$ s.t., for all $N \in \N$,
        \begin{equation}\label{eq:lem:helper:0}
                \AWD(\mu,\hat{\mu}^{N}) \leq C\Big(\Delta_N  + \WD(\mu_{1},\mu^{N}_{1}) + \sum_{t=1}^{T-1}\sum_{G\in\hat{\Phi}_{t}^{N}}\mu^{N}(G)\WD(\mu_{G},\mu_{G}^{N})\Big);
        \end{equation}
        \item there exists $C > 0$ s.t., for all $N \in \N$,
        \begin{equation}\label{eq:lem:helper:1}
                \AWD(\mu,\check{\mu}^{N}) \leq C\Big(\Delta_N  + \WD(\mu_{1},\mu^{N}_{1}) + \sum_{t=1}^{T-1}\sum_{G\in\check{\Phi}_{t}^{N}}\mu^{N}(G)\WD(\mu_{G},\mu_{G}^{N})\Big)+ \frac{C\Delta_N}{N}\sum_{n=1}^{N}\Vert X^{(n)} \Vert.
        \end{equation}
    \end{enumerate}
\end{lemma}
\begin{proof}
    \textbf{Step 1:} In this step, all statements hold for both uniform and non-uniform adapted empirical measures. Recall that we use the notation $\bar{\mu}^{N}$ to denote the uniform or non-uniform adapted empirical measure.
    \add{In \cite{backhoff2022estimating}, Lemma 3.1 proves that, when $\mu \in \calP([0,1]^{dT})$,}
    \begin{align}\label{eq:lem:helper:2}
        \AWD(\mu,\bar{\mu}^{N}) &\leq C\WD(\mu_1,\bar{\mu}^{N}_1) + C\sum_{t=1}^{T-1}\sum_{G\in\Phi_{t}^{N}}\int_{G}\WD(\mu_{x_{1:t}},\bar{\mu}^{N}_{x_{1:t}})\bar{\mu}^{N}(dx_{1:t}).
    \end{align}
    Since the proof does not depend on the compactness, and therefore could be generalized to the case when $\mu \in \calP(\R^{dT})$. Now, by the triangle inequality we have
    \begin{equation}\label{eq:lem:helper:3}
        \WD(\mu_1,\bar{\mu}^{N}_1) \leq \WD(\mu_1,\mu^{N}_1) + \WD(\mu^{N}_1,\bar{\mu}^{N}_1).
    \end{equation}
    Also, noticing that $\mu^{N}_1(G) = \bar{\mu}^{N}_1(G)$ for all $G \in \Phi_{1}^{N}$, we have that 
    \begin{equation}\label{eq:lem:helper:4}
        \begin{aligned}
            \WD(\mu^{N}_1,\bar{\mu}^{N}_1) &\leq \sum_{G \in \Phi_{1}^{N}}\mu^{N}(G)\WD\big(\mu^{N}_1\vert_G,\bar{\mu}^{N}_1\vert_G\big) \leq \sum_{G \in \Phi_{1}^{N}}\mu^{N}(G)\mathrm{diam}(G).
        \end{aligned}
    \end{equation}
    By combining \eqref{eq:lem:helper:3} and \eqref{eq:lem:helper:4}, we obtain
    \begin{equation}\label{eq:lem:helper:5}
        \begin{aligned}
            \WD(\mu_1,\bar{\mu}^{N}_1) \leq \WD(\mu_1,\mu^{N}_1) + \sum_{G \in \Phi_{1}^{N}}\mu^{N}(G)\mathrm{diam}(G).
        \end{aligned}
    \end{equation}
    Next, we estimate the second term in \eqref{eq:lem:helper:2} by using the Lipschitz kernel assumption: for any $G \in \Phi_{1}^{N}$,
    \begin{equation}\label{eq:lem:helper:6}
        \begin{split}
            \int_G\WD(\mu_{x_1},(\bar{\mu}^{N})_{x_1})\bar{\mu}^{N}(dg) 
            &\leq \int_G \Big(\WD(\mu_{x_1},\mu_G) + \WD(\mu_G,\mu_G^N) + \WD(\mu_G^N,(\bar{\mu}^{N})_{x_1})\Big)\bar{\mu}^{N}(dg)  \\
            &\leq \int_G \Big(L\cdot\mathrm{diam}(G) + \WD(\mu_G,\mu_G^N) + \WD(\mu_G^N,\bar{\mu}_G^N)\Big) \bar{\mu}^{N}(dg) \\
            &\leq  \mu^{N}(G)\Big(L\cdot \mathrm{diam}(G) + \WD(\mu_{G},\mu_{G}^{N}) + \WD(\mu_G^N,\bar{\mu}_G^N)\Big).
        \end{split}
    \end{equation}
    Combining \eqref{eq:lem:helper:2}, \eqref{eq:lem:helper:5} and \eqref{eq:lem:helper:6}, we conclude for both uniform and non-uniform adapted empirical measures that there exists $C > 0$ such that 
    \begin{equation}\label{eq:lem:helper:7}
            \AWD(\mu,\bar{\mu}^{N}) \leq C\WD(\mu_1,\mu^{N}_1) + C\sum_{t=1}^{T-1}\sum_{G\in\Phi_{t}^{N}}\mu^{N}(G)\Big(\WD(\mu_{G},\mu_{G}^{N}) + \WD(\mu_G^N,\bar{\mu}_G^N)\Big)   + C\sum_{t=1}^{T-1}\sum_{G\in\Phi_{t}^{N}}\mu^{N}(G)\mathrm{diam}(G) .
    \end{equation}
    \textbf{Step 2:} In this step, we separate the proof between uniform and non-uniform adapted empirical measures.\\
    (i) Recall that we denote by $\hat{\mu}^{N}$ the uniform adapted empirical measure, then by \eqref{eq:estimate_helper:1} we have for all $t = 1,\dots,T$ that 
    \begin{equation}\label{eq:lem:helper:8}
        \sum_{G\in\hat{\Phi}_{t}^{N}}\mu^{N}(G)\mathrm{diam}(G) \leq t\sqrt{d}\Delta_N\quad \text{and}\quad \sum_{G\in\hat{\Phi}_{t}^{N}}\mu^{N}(G)\WD(\mu_G^N,\hat{\mu}_G^N) \leq \sqrt{d} \Delta_N.
    \end{equation}
    Combining \eqref{eq:lem:helper:7} and \eqref{eq:lem:helper:8}, we conclude that there exists $C > 0$ such that 
    \begin{equation}\label{eq:lem:helper:9}
        \begin{split}
            \AWD(\mu,\hat{\mu}^{N}) &\leq C\Big(\Delta_N  + \WD(\mu_{1},\mu^{N}_{1}) + \sum_{t=1}^{T-1}\sum_{G\in\hat{\Phi}_{t}^{N}}\mu^{N}(G)\WD(\mu_{G},\mu_{G}^{N})\Big).
        \end{split}
    \end{equation}
    (ii) Recall that we denote by $\check{\mu}^{N}$ the non-uniform adapted empirical measure, then by \eqref{eq:estimate_helper:2} we have for all $t = 1,\dots,T$ that 
    \begin{equation}\label{eq:lem:helper:10}
        \begin{split}
            \sum_{G\in\check{\Phi}_{t}^{N}}\mu^{N}(G)\mathrm{diam}(G)&= \sum_{j=0}^{\infty} \sum_{G \in \check{\Phi}_{t}^{N,j}}\mathrm{diam}(G)\mu^{N}(G)\leq  \sum_{j=0}^{\infty} \sum_{G \in \check{\Phi}_{t}^{N,j}}t\sqrt{d}\Delta_N 2^{j}\mu^{N}(G) = t\sqrt{d}\Delta_N\sum_{j=0}^{\infty} 2^{j}\mu^{N}(\mathcal{A}_{j}^{t}),
        \end{split}
    \end{equation}
    and 
    \begin{equation}\label{eq:lem:helper:11}
        \begin{split}
            \sum_{G\in\check{\Phi}_{t}^{N}}\mu^{N}(G)\WD(\mu_G^N,\check{\mu}_G^N) &\leq \sum_{G\in\check{\Phi}_{t}^{N}}\mu^{N}(G)\sum_{j=0}^{\infty}\sum_{G' \in \check{\Phi}_{1}^{N,j}}\mathrm{diam}(G')\mu^{N}_{G}(G') \leq \sum_{G\in\check{\Phi}_{t}^{N}}\mu^{N}(G)\sum_{j=0}^{\infty}\sum_{G' \in \check{\Phi}_{1}^{N,j}}t\sqrt{d}\Delta_N 2^{j}\mu^{N}_{G}(G')\\
            &= t\sqrt{d}\Delta_N \sum_{j=0}^{\infty}\sum_{G\in\check{\Phi}_{t}^{N}}\mu^{N}(G)\mu^{N}_{G}(\mathcal{A}_{j}^{1}) 2^j 
            \leq t\sqrt{d}\Delta_N\sum_{j=0}^{\infty} 2^{j}\mu_{t+1}^{N}(\mathcal{A}_{j}^{1}).
        \end{split}
    \end{equation}
    Noticing that $\min_{x \in \mathcal{A}_{j}^{t}}\Vert x \Vert \geq 2^{j-1}$ for all $j\in\N$ and $t=1,\dots,T-1$, we have
    \begin{equation}\label{eq:lem:helper:12}
        \begin{split}
        t\sqrt{d}\Delta_N\sum_{j=0}^{\infty} 2^{j}\mu^{N}(\mathcal{A}_{j}^{t})&\leq t\sqrt{d}\Delta_N + 2t\sqrt{d}\Delta_N\sum_{j=1}^{\infty} 2^{j-1} \mu^{N}(\mathcal{A}_{j}^{t})
        \leq t\sqrt{d}\Delta_N +  2t\sqrt{d}\Delta_N\int_{\R^{dT}} \Vert x_{1:t} \Vert \mu^{N}(dx)\\
        &\leq T\sqrt{d}\Delta_N +  2T\sqrt{d}\Delta_N\int_{\R^{dT}} \Vert x \Vert \mu^{N}(dx) =  T\sqrt{d}\Delta_N + \frac{2T\sqrt{d}\Delta_N}{N}\sum_{n=1}^{N}\Vert X^{(n)} \Vert,\\
        t\sqrt{d}\Delta_N\sum_{j=0}^{\infty} 2^{j}\mu^{N}_{t+1}(\mathcal{A}_{j}^{1}) &\leq t\sqrt{d}\Delta_N + 2t\sqrt{d}\Delta_N\sum_{j=1}^{\infty} 2^{j-1} \mu^{N}_{t+1}(\mathcal{A}_{j}^{1}) \leq t\sqrt{d}\Delta_N +  2t\sqrt{d}\Delta_N\int_{\R^{dT}} \Vert x_{t+1} \Vert \mu^{N}(dx)\\
        &\leq T\sqrt{d}\Delta_N +  2T\sqrt{d}\Delta_N\int_{\R^{dT}} \Vert x \Vert \mu^{N}(dx) =  T\sqrt{d}\Delta_N + \frac{2T\sqrt{d}\Delta_N}{N}\sum_{n=1}^{N}\Vert X^{(n)} \Vert.
        \end{split}
    \end{equation}
    Combining \eqref{eq:lem:helper:7}, \eqref{eq:lem:helper:10}, \eqref{eq:lem:helper:11} and \eqref{eq:lem:helper:12}, we conclude that there exists $C > 0$ such that, for all $N \in \N$,
    \begin{equation}\label{eq:lem:helper:13}
            \AWD(\mu,\check{\mu}^{N}) \leq C\Big(\Delta_N  + \WD(\mu_{1},\mu^{N}_{1}) + \sum_{t=1}^{T-1}\sum_{G\in\check{\Phi}_{t}^{N}}\mu^{N}(G)\WD(\mu_{G},\mu_{G}^{N})\Big) + \frac{C\Delta_N}{N}\sum_{n=1}^{N}\Vert X^{(n)} \Vert.
    \end{equation}
    Therefore, \eqref{eq:lem:helper:9} and \eqref{eq:lem:helper:13} together complete the proof of Lemma \ref{lem:helper}.
\end{proof}
In the next Lemma, we show that, conditioning on the adapted empirical law up to time $t$, the average kernel of the empirical law is the same as the empirical law of the average kernel.
\begin{lemma}\label{lem:helper2}
    Let $E$ be a Borel set in $\mathcal{P}(\R^d)$. Then, for all $G \in \Phi_{t}^{N}$,
    \begin{equation}
    \label{eq:lem:helper2}
        \P\Big[\mu_{G}^{N} \in E \Big\vert \mu^{N}(G'), \forall G'\in\Phi_{t}^{N} \Big] =  \P\Big[\mu_{G}^{N} \in E \Big\vert \mu^{N}(G) \Big] = \P\big[(\mu_{G})^{N\mu^{N}(G)} \in E \Big],
    \end{equation}
    where $(\mu_{G})^{N\mu^{N}(G)}$ is the empirical measure of $\mu_{G}$ with sample size $N\mu^{N}(G) \in \N$. \add{Moreover, $(\mu_G^N)_{G\in\Phi^N_t}$ are independent given $\{\mu^{N}(G') \colon G'\in\Phi_{t}^{N}\}$.}
\end{lemma}
\begin{proof}
    \add{
    \eqref{eq:lem:helper2} follows from the proof of Lemma 3.3 in \cite{backhoff2022estimating}, where $\mu \in \calP([0,1]^{dT})$. Indeed, that proof does not depend on the compactness of the support of $\mu$, and can be generalized to any $\mu \in \calP(\R^{dT})$.}
\end{proof}
We conclude this preparatory section by estimating moments and exponential moments of shifted and scaled $\mu_G$. 
\begin{lemma}
    \label{lem:shift_amp}
    Let $\mu \in \mathcal{P}(\R^{dT})$.
    \begin{enumerate}[(i)]
        \item Assume Setting \ref{setting:1} or Setting \ref{setting:2}. Then, for all $t=1,\dots,T-1$, $G \in \Phi^{N}_t$, there exist $f_G \in \R^d$ and $\sigma_G,\lambda_G>0$ s.t. 
        \begin{equation}
        \label{eq:lem:shift_amp:0.1}
            M_p^{1/p}\Big(\frac{1}{\lambda_G}\frac{\mu_G - f_G}{\sigma_G}\Big) \leq 1,
        \end{equation}
        where $\sigma_G = \sup_{x_{1:t}\in G}\sigma_t(x_{1:t}) \vee 1$ and $\lambda_G = L_f \cdot \mathrm{diam}(G) + \big(\sup_{x_{1:t}\in \R^{dt}}M_{p}(\lambda_{x_{1:t}})\big)^{1/p}$.
        \item Assume Setting \ref{setting:3}. Then, for all $t=1,\dots,T-1$, $G \in \Phi^{N}_t$, there exist $f_G \in \R^d$ and $\sigma_G,\lambda_G>0$ s.t. 
        \begin{equation}
        \label{eq:lem:shift_amp:0.2}
        \mathcal{E}_{1,\alpha}\Big(\frac{1}{\lambda_G}\frac{\mu_G - f_G}{\sigma_G}\Big) \leq e \cdot \Big(\sup_{x_{1:t}\in \R^{dt}}\mathcal{E}_{\gamma,\alpha}(\lambda_{x_{1:t}})\Big),
        \end{equation}
        where $\sigma_G = \sup_{x_{1:t}\in G}\sigma_t(x_{1:t}) \vee 1$ and $\lambda_G = 2(L_f\mathrm{diam}(G)+1) + 2\gamma^{-\alpha}$.
    \end{enumerate}
\end{lemma}
\begin{proof}
    Recall that for all $t=1,\dots,T-1$, $G\in\Phi^N_t$, $f_G \in \R^{d}$, $\sigma_G > 0$, 
    \begin{equation*}
        \frac{\mu_G - f_G}{\sigma_G} = \int_{\R^{dt}} \frac{\mu_{x_{1:t}} - f_G}{\sigma_G}\mu\vert_G(dx_{1:t})
        =\int_{\R^{dt}} \frac{f_{t}(x_{1:t}) - f_G + \sigma_t(x_{1:t})\lambda_{x_{1:t}}}{\sigma_G}\mu\vert_G(dx_{1:t}).
    \end{equation*}    
    (i) First, we estimate the $p$-th moment under Setting~\ref{setting:1} or Setting~\ref{setting:2}. In this case,
    \begin{equation*}
    \begin{split}
        M_p^{1/p}\left(\frac{\mu_G - f_G}{\sigma_G}\right) &= \left(\int_{\R^{d}}\Vert x_{t+1}\Vert^p \left(\int_{\R^{dt}} \frac{f_{t}(x_{1:t}) - f_G + \sigma_t(x_{1:t})\lambda_{x_{1:t}}}{\sigma_G}\mu\vert_G(dx_{1:t})\right)(dx_{t+1})\right)^{1/p}\\
        &= \left(\int_{\R^{dt}}\int_{\R^{d}}\Vert x_{t+1}\Vert^p  \frac{f_{t}(x_{1:t}) - f_G + \sigma_t(x_{1:t})\lambda_{x_{1:t}}}{\sigma_G}(dx_{t+1})\mu\vert_G(dx_{1:t})\right)^{1/p}\\
        &= \left(\int_{\R^{dt}}\E\left[\left\Vert \frac{f_{t}(x_{1:t}) - f_G + \sigma_t(x_{1:t})\epsilon_{x_{1:t}}}{\sigma_G}\right\Vert^p  \right]\mu\vert_G(dx_{1:t})\right)^{1/p}\\
        &\leq \left(\int_{\R^{dt}}\left\Vert \frac{f_{t}(x_{1:t}) - f_G}{\sigma_G}\right\Vert^p  \mu\vert_G(dx_{1:t})\right)^{1/p} + \left(\int_{\R^{dt}}\E\left[\left\Vert \frac{\sigma_t(x_{1:t})\epsilon_{x_{1:t}}}{\sigma_G}\right\Vert^p  \right]\mu\vert_G(dx_{1:t})\right)^{1/p}.
    \end{split}
    \end{equation*}
    By choosing $x_G$ as the mid point of the cube $G$, $f_G = f_t(x_G)$ and $\sigma_G = \sup_{x_{1:t}\in G}\sigma_t(x_{1:t}) \vee 1$, we have
    \begin{equation*}
        \begin{split}
            M_p^{1/p}\left(\frac{\mu_G - f_G}{\sigma_G}\right) &\leq \left(\int_{\R^{dt}}\left\Vert f_{t}(x_{1:t}) - f_G\right\Vert^p  \mu\vert_G(dx_{1:t})\right)^{1/p} + \left(\int_{\R^{dt}}\E\left[\left\Vert \epsilon_{x_{1:t}}\right\Vert^p  \right]\mu\vert_G(dx_{1:t})\right)^{1/p}\\
            &\leq L_f \cdot \mathrm{diam}(G) + \Big(\sup_{x_{1:t}\in \R^{dt}}M_{p}(\lambda_{x_{1:t}})\Big)^{1/p} < \infty.
        \end{split}
    \end{equation*}
    Therefore, by choosing $\lambda_G = L_f \cdot \mathrm{diam}(G) + \Big(\sup_{x_{1:t}\in \R^{dt}}M_{p}(\lambda_{x_{1:t}})\Big)^{1/p}$, we get \eqref{eq:lem:shift_amp:0.1}.\\
    (ii) Next, we estimate the exponential moment under Setting~\ref{setting:3}. For any $\tilde{\gamma} > 0$,
    \begin{equation*}
    \begin{split}
        \mathcal{E}_{\tilde{\gamma},\alpha}\left(\frac{\mu_G - f_G}{\sigma_G}\right) 
        &= \int_{\R^{dt}}\int_{\R^{d}}e^{\tilde{\gamma} \Vert x_{t+1}\Vert^\alpha}  \frac{f_{t}(x_{1:t}) - f_G + \sigma_t(x_{1:t})\lambda_{x_{1:t}}}{\sigma_G}(dx_{t+1})\mu\vert_G(dx_{1:t})\\
        &= \int_{\R^{dt}}\E\left[\exp\left\{\tilde{\gamma} \left\Vert \frac{f_{t}(x_{1:t}) - f_G + \sigma_t(x_{1:t})\epsilon_{x_{1:t}}}{\sigma_G}\right\Vert^\alpha \right\} \right]\mu\vert_G(dx_{1:t})\\
        &\leq  \int_{\R^{dt}}\exp\left\{2^{\alpha}\tilde{\gamma}\left\Vert \frac{f_{t}(x_{1:t}) - f_G}{\sigma_G}\right\Vert^\alpha\right\}\E\left[\exp\left\{2^{\alpha}\tilde{\gamma}\left\Vert \frac{\sigma_t(x_{1:t})\epsilon_{x_{1:t}}}{\sigma_G}\right\Vert^\alpha\right\}  \right]\mu\vert_G(dx_{1:t}).
    \end{split}
    \end{equation*}
    By choosing $x_G$ as the mid point of $G$, $f_G = f_t(x_G)$ and $\sigma_G = \sup_{x_{1:t}\in G}\sigma_t(x_{1:t}) \vee 1$ as above, we have
    \begin{equation}
    \label{eq:lem:shift_amp:1}
        \begin{split}
        \mathcal{E}_{\tilde{\gamma},\alpha}\left(\frac{\mu_G - f_G}{\sigma_G}\right) &\leq \int_{\R^{dt}}\exp\{2^{\alpha}\tilde{\gamma}\left\Vert f_{t}(x_{1:t}) - f_G\right\Vert^\alpha\}\E\left[e^{2^{\alpha}\tilde{\gamma}\left\Vert \epsilon_{x_{1:t}}\right\Vert^\alpha}  \right]\mu\vert_G(dx_{1:t})\\
        &\leq \exp\{2^{\alpha}\tilde{\gamma}  L_f^{\alpha} \mathrm{diam}(G)^{\alpha}\}\Big(\sup_{x_{1:t}\in \R^{dt}}\mathcal{E}_{2^{\alpha}\tilde{\gamma},\alpha}(\lambda_{x_{1:t}})\Big)< \infty.
        \end{split}
    \end{equation}
    Finally, by choosing $\lambda_G = 2L_f\mathrm{diam}(G) + 2\gamma^{-\alpha}$ and $\tilde{\gamma} = \lambda_G^{-\alpha}$ in \eqref{eq:lem:shift_amp:1}, we have
    \begin{equation*}
    \begin{split}
        \mathcal{E}_{1,\alpha}\left(\frac{1}{\lambda_G}\frac{\mu_G - f_G}{\sigma_G}\right) &= \mathcal{E}_{\frac{1}{\lambda_G^{\alpha}},\alpha}\left(\frac{\mu_G - f_G}{\sigma_G}\right) \leq \exp\{\lambda_G^{-\alpha}2^{\alpha} L_f^{\alpha} \mathrm{diam}(G)^{\alpha}\} \Big(\sup_{x_{1:t}\in \R^{dt}}\mathcal{E}_{2^{\alpha}\lambda_G^{-\alpha},\alpha}(\lambda_{x_{1:t}})\Big)\\
        &\leq e \Big(\sup_{x_{1:t}\in \R^{dt}}\mathcal{E}_{\gamma,\alpha}(\lambda_{x_{1:t}})\Big) < \infty,
    \end{split}
    \end{equation*}
    which proves \eqref{eq:lem:shift_amp:0.2}.
\end{proof}

\section{Moment estimate} \label{sect:mom_est}
This section is dedicated to the proof of Theorem \ref{thm:moment_estimate}.
From Lemma~\ref{lem:helper} we see that essentially we need to estimate $\sum_{G\in\Phi_{t}^{N}}\mu^{N}(G)\WD(\mu_{G},\mu_{G}^{N})$. Since this estimation will be used multiple times in the paper, we state the key parts of estimation as two lemmas below. 

Recall that, for notational simplicity,  we present the proof when $d \geq 3$.
\label{sec:moment_estimate} 
\begin{lemma}\label{lem:mean_shift_amp}
Let $\mu \in \mathcal{P}(\R^{dT})$.
\begin{enumerate}[(i)]
    \item Assume Setting \ref{setting:1}. Then there exists $C>0$ s.t., for all $t=1,\dots,T-1$,
    \begin{equation*}
    \begin{split}
        \E\Big[\sum_{G\in \hat{\Phi}_t^{N}}\mu^{N}(G)\WD(\mu_{G},\mu_{G}^{N}) \Big]\leq C \E\Big[ \sum_{G\in \hat{\Phi}^{N}_t}\mu^{N}(G) (\Vert G\Vert^r +1)\Big((N\mu^N(G))^{-\frac{1}{d}} + (N\mu^N(G))^{-\frac{p-1}{p}}\Big)\Big].
    \end{split}
    \end{equation*}
    \item Assume Setting \ref{setting:2}. Then there exists $C>0$ s.t., for all $t=1,\dots,T-1$,
    \begin{equation*}
    \begin{split}
        \E\Big[\sum_{G\in \check{\Phi}_t^{N}}\mu^{N}(G)\WD(\mu_{G},\mu_{G}^{N}) \Big]\leq C \E\Big[ \sum_{G\in \check{\Phi}^{N}_t}\mu^{N}(G) (\Vert G\Vert^r +1)\Big((N\mu^N(G))^{-\frac{1}{d}} + (N\mu^N(G))^{-\frac{p-1}{p}}\Big)\Big] + C\Delta_N.
    \end{split}
    \end{equation*}
\end{enumerate}
\end{lemma}
\begin{proof}
    First notice that, by shifting and scaling of $\WD$, for all $t=1,\dots,T-1$, $G\in\Phi^N_t$, $f_G \in \R^{d}$, $\sigma_G, \lambda_G > 0$ and $N_G \in \N$,
    \begin{equation}
    \label{eq:lem:mean_shift_amp:1}
        \WD(\mu_G,(\mu_G)^{N_G}) = \sigma_G\lambda_G\WD\bigg(\frac{1}{\lambda_G}\frac{\mu_G - f_G}{\sigma_G}, \Big(\frac{1}{\lambda_G}\frac{\mu_G - f_G}{\sigma_G}\Big)^{N_G}\bigg).
    \end{equation}
    Next, by Lemma~\ref{lem:shift_amp}, we know that there exist $f_G \in \R^{d}$ and $\sigma_G, \lambda_G > 0$ s.t. $M_p^{1/p}\Big(\frac{1}{\lambda_G}\frac{\mu_G - f_G}{\sigma_G}\Big) \leq 1$, so we can apply Theorem \ref{thm:wass_rate} to the measure $\frac{1}{\lambda_G}\frac{\mu_G - f_G}{\sigma_G}$. Thus, there exists $C >0$ s.t., for all $N_G \in \N$,
    \begin{equation}
    \label{eq:lem:mean_shift_amp:2}
    \begin{split}
        \WD\bigg(\frac{1}{\lambda_G}\frac{\mu_G - f_G}{\sigma_G}, \Big(\frac{1}{\lambda_G}\frac{\mu_G - f_G}{\sigma_G}\Big)^{N_G}\bigg) \leq C\Big(N_G^{-\frac{1}{d}} + N_G^{-\frac{p-1}{p}}\Big).
    \end{split}
    \end{equation}
    Note that, by the choice of $\sigma_G$ in Lemma~\ref{lem:shift_amp},  $\sigma_G \leq \sup_{x_{1:t}\in G}\sigma_t(x_{1:t}) + 1 \leq c\Vert G\Vert^r + 1$ because of the growth of $\sigma_t$ in Setting~\ref{setting:1} and Setting~\ref{setting:2}. Combining this with \eqref{eq:lem:mean_shift_amp:1} and \eqref{eq:lem:mean_shift_amp:2}, we obtain that there exists $C>0$ s.t., for all $N_G \in \N$,
    \begin{equation*}
        \WD(\mu_G,(\mu_G)^{N_G}) \leq C  (\Vert G\Vert^r+1)\left(\mathrm{diam}(G) + 1\right)\Big(N_G^{-\frac{1}{d}} + N_G^{-\frac{p-1}{p}}\Big).
    \end{equation*}
    Together with Lemma~\ref{lem:helper2}, this implies
    \begin{equation}
    \label{eq:lem:mean_shift_amp:3}
    \begin{split}
    &\quad\E\Big[\sum_{G\in \Phi^{N}_t}\mu^{N}(G)\WD(\mu_{G},\mu_{G}^{N}) \Big\vert \mu^{N}(G'), \forall G'\in\Phi_{t}^{N}\Big] \\
    &= \sum_{G\in \Phi^{N}_t}\mu^{N}(G)\E\left[\WD(\mu_G,(\mu_G)^{N\mu^N(G)})\Big\vert \mu^{N}(G'), \forall G'\in\Phi_{t}^{N}\right]\\
    &\leq C  \sum_{G\in \Phi^{N}_t}\mu^{N}(G)(\Vert G\Vert^r+1)\left(\mathrm{diam}(G) + 1\right)\mathrm{rate}_G(N),
    \end{split}
    \end{equation}
    where $\mathrm{rate}_G(N) = \Big((N\mu^N(G))^{-\frac{1}{d}} + (N\mu^N(G))^{-\frac{p-1}{p}}\Big)$. Therefore, by taking expectation of \eqref{eq:lem:mean_shift_amp:3}, we conclude that there exists $C > 0$ s.t. 
    \begin{equation}
    \label{eq:lem:mean_shift_amp:4}
    \begin{split}
        \E\Big[\sum_{G\in \Phi^{N}_t}\mu^{N}(G)\WD(\mu_{G},\mu_{G}^{N}) \Big] &\leq C  \E\Big[\sum_{G\in \Phi^{N}_t}\mu^{N}(G) (\Vert G\Vert^r +1)(\mathrm{diam}(G)+1)\,\mathrm{rate}_G(N)\Big].
    \end{split}
    \end{equation}
    (i)
    In the uniform case, for all $G\in\hat{\Phi}_t^{N}$, $\mathrm{diam}(G) =t\sqrt{d}\Delta_N\leq 1$. Thus, there exists $C > 0$ s.t. 
    \begin{equation*}
    \begin{split}
        \E\Big[\sum_{G\in \hat{\Phi}_t^{N}}\mu^{N}(G)\WD(\mu_{G},\mu_{G}^{N}) \Big] \leq C \E\Big[ \sum_{G\in \hat{\Phi}^{N}_t}\mu^{N}(G) (\Vert G\Vert^r +1)\,\mathrm{rate}_G(N)\Big].
    \end{split}
    \end{equation*}
    (ii) In the non-uniform case, notice that
    \begin{equation}
    \label{eq:lem:mean_shift_amp:5}
        \begin{split}
        \E\Big[\sum_{G\in \check{\Phi}^{N}_t}\mu^{N}(G) (\Vert G\Vert^r +1)\mathrm{diam}(G)\,\mathrm{rate}_G(N)\Big]\leq 2\sum_{G\in \check{\Phi}^{N}_t}\mu(G) (\Vert G\Vert^r +1)\mathrm{diam}(G).
        \end{split}
    \end{equation}
    Moreover, recall that for all $t=1,\dots,T-1$, $j\geq 0$ and $G \in \check{\Phi}^{N,j}_t$, we have $\mathrm{diam}(G) = t\sqrt{d}2^{j}\Delta_N$, $\Vert G \Vert \leq t\sqrt{d}2^j$ and $\mu(\mathcal{A}^{t}_{j})\leq \frac{M_{q}(\mu_{1:t})}{2^{q(j-1)}}$. Thus, there exists $C >0$ s.t.
    \begin{equation}
    \label{eq:lem:mean_shift_amp:6}
    \begin{split}
    \sum_{G\in \check{\Phi}^{N}_t}\mu(G) (\Vert G\Vert^r +1)\mathrm{diam}(G) &= \sum_{j\geq 0}\sum_{G\in \check{\Phi}^{N,j}_t}\mu(G) (\Vert G\Vert^r +1)\mathrm{diam}(G)\\
    &\leq \Delta_N\sum_{j\geq 0}\frac{M_{q}(\mu_{1:t})}{2^{q(j-1)}} ((t\sqrt{d}2^{j})^r+1)t\sqrt{d}2^{j} \leq C \Delta_N,
    \end{split}
    \end{equation}
    because $q > 1+ r$ in Setting~\ref{setting:2}. Combining \eqref{eq:lem:mean_shift_amp:4}, \eqref{eq:lem:mean_shift_amp:5} and \eqref{eq:lem:mean_shift_amp:6}, we obtain that there exists $C>0$ s.t. 
    \begin{equation*}
    \begin{split}
        \E\Big[\sum_{G\in \check{\Phi}_t^{N}}\mu^{N}(G)\WD(\mu_{G},\mu_{G}^{N}) \Big] \leq C \E\Big[ \sum_{G\in \check{\Phi}^{N}_t}\mu^{N}(G) (\Vert G\Vert^r +1)\,\mathrm{rate}_G(N)\Big] + C\Delta_N.
    \end{split}
    \end{equation*}
\end{proof}
\begin{lemma}
\label{lem:common_estimate}
Let $\mu\in\calP(\R^{dT})$, $a \in [0,1)$, $b > 0$, $c>0$ and $r\geq 0$. Then there exists $C>0$ s.t., for all $t=1,\dots,T-1$,
\begin{equation*}
    \left\{\begin{aligned}
        \sum_{G\in\hat{\Phi}^N_t}(\Vert G \Vert^r + 1)^{b}(\mathrm{diam}(G)+1)^{c}\mu(G)^{1-a} &\leq C N^{\frac{ta}{T}}\sum_{j\geq 0}2^{j\big(rb + dta\big)}\mu(\calA_j)^{1-a},\\
        \sum_{G\in\check{\Phi}^N_t}(\Vert G \Vert^r + 1)^{b}(\mathrm{diam}(G)+1)^{c}\mu(G)^{1-a} &\leq C N^{\frac{ta}{T}}\sum_{j\geq 0}2^{j\big(rb + c\big)}\mu(\calA_j)^{1-a}.
    \end{aligned}\right.
\end{equation*}
\end{lemma}
\begin{proof}
For all $t=1,\dots,T-1$, $j\geq 0$ and $G \in \hat{\Phi}^{N,j}_t$, we have $\mathrm{diam}(G) = t\sqrt{d}N^{\frac{1}{dT}}$ , $\Vert G \Vert \leq t\sqrt{d}2^j$ and $\vert \hat{\Phi}^{N,j}_t \vert \leq (2^{j+1}N^{\frac{1}{dT}})^{dt}$. Thus, there exists $C >0$ s.t.
\begin{equation*}
        \begin{split}
            &\quad~\sum_{G\in\hat{\Phi}^N_t}(\Vert G \Vert^r + 1)^{b}(\mathrm{diam}(G)+1)^{c}\mu(G)^{1-a} \leq C \sum_{j=0}^{\infty}\sum_{G\in \hat{\Phi}^{N,j}_{t}}(\Vert G \Vert^r+1)^{b}\mu(G)^{1-a} \\
            &\leq C \sum_{j=0}^{\infty}\sum_{G\in \hat{\Phi}^{N,j}_{t}}2^{rb j}\mu(G)^{1-a} \leq C \sum_{j=0}^{\infty}2^{rb j}\sum_{G\in \hat{\Phi}^{N,j}_{t}}\left(\frac{\mu(\calA_j)}{\vert \hat{\Phi}^{N,j}_t \vert}\right)^{1-a} \\
            &= C \sum_{j=0}^{\infty}2^{rb j}\vert \hat{\Phi}^{N,j}_t \vert^{a}\mu(\calA_j)^{1-a} \leq C \sum_{j=0}^{\infty}2^{rb j}\cdot  N^{\frac{ta}{T}} 2^{jdta} \cdot \mu(\calA_j)^{1-a} \leq C N^{\frac{ta}{T}}\sum_{j\geq 0}2^{j\big(rb + dta\big)}\mu(\calA_j)^{1-a}.
        \end{split}
    \end{equation*}
For all $t=1,\dots,T-1$, $j\geq 0$ and $G \in \check{\Phi}^{N,j}_t$, we have $\mathrm{diam}(G) = t\sqrt{d}2^j N^{\frac{1}{dT}}$ , $\Vert G \Vert \leq t\sqrt{d}2^j$ and $\vert \check{\Phi}^{N,j}_t \vert \leq (4N^{\frac{1}{dT}})^{dt}$. Thus, there exists $C >0$ s.t.
\begin{equation*}
        \begin{split}
            &\quad~\sum_{G\in\check{\Phi}^N_t}(\Vert G \Vert^r + 1)^{b}(\mathrm{diam}(G)+1)^{c}\mu(G)^{1-a} \leq C \sum_{j=0}^{\infty}\sum_{G\in \check{\Phi}^{N,j}_{t}}2^{rb j} \cdot 2^{c j} \cdot \mu(G)^{1-a} \\
            &\leq C \sum_{j=0}^{\infty}\sum_{G\in \check{\Phi}^{N,j}_{t}}2^{(rb + c)j}\mu(G)^{1-a} \leq C \sum_{j=0}^{\infty}2^{(rb +c)j}\sum_{G\in \check{\Phi}^{N,j}_{t}}\left(\frac{\mu(\calA_j)}{\vert \check{\Phi}^{N,j}_t \vert}\right)^{1-a} \\
            &= C \sum_{j=0}^{\infty}2^{(rb+c) j}\vert \check{\Phi}^{N,j}_t \vert^{a}\mu(\calA_j)^{1-a} \leq C \sum_{j=0}^{\infty}2^{(rb+c)j}\cdot  N^{\frac{ta}{T}} \cdot \mu(\calA_j)^{1-a} \leq C N^{\frac{ta}{T}}\sum_{j\geq 0}2^{j\big(rb + c\big)}\mu(\calA_j)^{1-a}.
        \end{split}
    \end{equation*}
\end{proof}
Before proving our moment estimates for the adapted Wasserstein distance, we recall the classical Wasserstein distance case.
\begin{theorem}[Theorem 1 in \cite{fournier2015rate}]
    \label{thm:wass_rate}
    Let $\mu \in \mathcal{P}(\R^d)$ and assume there exists $p > 1$ s.t. $M_{p}(\mu) < \infty$. Then there exists $C >0$ s.t., for all $N \in \N$,
    \begin{equation*}
        \E\big[\WD_{1}(\mu,\mu^{N})\big] \leq CM_{p}^{1/p}(\mu)\Big(N^{-\frac{1}{d}} + N^{-\frac{p-1}{p}}\Big).
    \end{equation*}
\end{theorem}
\begin{proof}[Proof of Theorem \ref{thm:moment_estimate}]
\textbf{Step 1:} In this step, all statements hold for both uniform and non-uniform adapted empirical measures. Taking expectation in \eqref{eq:lem:helper:0} and \eqref{eq:lem:helper:1}, we have that there exists $C > 0$ such that 
\begin{equation}\label{eq:thm:moment_estimate:1}
    \begin{split}
        \E[\AWD(\mu,\bar{\mu}^{N})] \leq C\E[\WD(\mu_1,\mu^{N}_1)] + C\sum_{t=1}^{T-1}\E\big[\sum_{G\in\Phi_{t}^{N}}\mu^{N}(G)\WD(\mu_{G},\mu_{G}^{N})\big] + C\Delta_N,
    \end{split}
\end{equation}
since $\frac{1}{N}\sum_{n=1}^{N}\E[X^{(n)}] = \int \Vert x \Vert \mu(dx) < \infty$. First, we estimate $\E[\WD(\mu_1,\mu^{N}_1)]$ in \eqref{eq:thm:moment_estimate:1} by applying Theorem~\ref{thm:wass_rate} to $\mu_1$, and we obtain that there exists $C > 0$ such that, for all $N \in \N$,
\begin{equation}\label{eq:thm:moment_estimate:2}
    \E\big[\WD(\mu_{1},\mu^{N}_{1})\big] \leq CN^{-\frac{1}{d}} + CN^{-\frac{q-1}{q}}.
\end{equation}
Then, we estimate $\E\big[\sum_{G\in \Phi^{N}_t}\mu^{N}(G)\WD(\mu_{G},\mu_{G}^{N})\big]$ in \eqref{eq:thm:moment_estimate:1} for all $t = 1,\dots, T-1$.By Lemma~\ref{lem:mean_shift_amp}, there exists $C>0$ s.t. 
    \begin{equation}
    \label{eq:thm:moment_estimate:3}
        \E\Big[\sum_{G\in \Phi_t^{N}}\mu^{N}(G)\WD(\mu_{G},\mu_{G}^{N}) \Big] \leq C \E\Big[ \sum_{G\in \Phi^{N}_t}\mu^{N}(G) (\Vert G\Vert^r +1)\Big((N\mu^{N}(G))^{-\frac{1}{d}} + (N\mu^{N}(G))^{-\frac{p-1}{p}}\Big)\Big] + C\Delta_N.
    \end{equation}
    Furthermore, since $x \mapsto x \Big(x^{-\frac{1}{d}} + x^{-\frac{p-1}{p}}\Big)$ is concave and $\E[N\mu^{N}(G)] = N\mu(G)$ for all $G \in \Phi_{t}^{N}$, by Jensen's inequality we have that 
    \begin{equation}\label{eq:thm:moment_estimate:4}
        \begin{split}
            &N^{-1}\E[\sum_{G\in \Phi^{N}_t}\add{(\Vert G \Vert^r+1)} (N\mu^{N}(G)) ((N\mu^{N}(G))^{-\frac{1}{d}} + (N\mu^{N}(G))^{-\frac{p-1}{p}}) ]\\
            &\qquad\qquad \leq N^{-\frac{1}{d}}\sum_{G\in \Phi^{N}_t}\add{(\Vert G \Vert^r+1)} \mu(G)^{1-\frac{1}{d}} + N^{-\frac{p-1}{p}}\sum_{G\in \Phi^{N}_t}\add{(\Vert G \Vert^r+1)}\mu(G)^{1-\frac{p-1}{p}}.  
        \end{split}
    \end{equation}
    \textbf{Step 2:} In this step, we separate the proof for the uniform adapted empirical measure in Setting \ref{setting:1} and for the non-uniform adapted empirical measure in Setting \ref{setting:2}. Notice that $\mu(\mathcal{A}^{t}_{j})\leq \frac{M_{q}(\mu_{1:t})}{2^{q(j-1)}}$. Then, by Lemma~\ref{lem:common_estimate} (with $b = 1, c = 0$), for all $a \in [0,1)$, there exists $C > 0$ s.t. 
    \begin{equation}
    \label{eq:thm:moment_estimate:5.1}
    \begin{split}
            &\sum_{G\in \hat{\Phi}^{N}_{t}}\add{(\Vert G \Vert^r+1)} \mu(G)^{1-a} \leq CN^{\frac{ta}{T}}\sum_{j = 0}^{\infty}2^{j(r + dta-q(1-a))},\\
            &\sum_{G\in \check{\Phi}^{N}_{t}}\add{(\Vert G \Vert^r+1)} \mu(G)^{1-a} \leq CN^{\frac{ta}{T}}\sum_{t=0}^{\infty}2^{j(r - q(1-a))}.
    \end{split}
    \end{equation}
    (i) Assume Setting \ref{setting:1}.
    Since  $q > \max\left\{\frac{d}{d-1}(r + T -1),rp + d(T-1)(p-1)\right\}$, then for $a = 1/d$ and $a = \frac{p-1}{p}$, $r + dta-q(1-a) < 0$ holds for all $t=1,\dots,T-1$, thus the first series in \eqref{eq:thm:moment_estimate:5.1} converges. \\
    (ii) Assume Setting \ref{setting:2}. Since  $q > \max\left\{\frac{rd}{d-1},rp \right\}$, then for $a = 1/d$ and $a = \frac{p-1}{p}$, $r -q(1-a) < 0$ holds for all $t=1,\dots,T-1$, thus the second series in \eqref{eq:thm:moment_estimate:5.1} converges. \\
    \textbf{Step 3:} In this step, all statements hold for both uniform and non-uniform adapted empirical measures. First, by \eqref{eq:thm:moment_estimate:5.1} in Step 2, there exists $C>0$ s.t., for all $t=1,\dots,T-1$,
    \begin{equation}
    \label{eq:thm:moment_estimate:6}
        \sum_{G\in \Phi^{N}_{t}}\add{(\Vert G \Vert^r+1)}\mu(G)^{1-a} \leq CN^{\frac{ta}{T}}.
    \end{equation}
    Combining \eqref{eq:thm:moment_estimate:3}, \eqref{eq:thm:moment_estimate:4} and \eqref{eq:thm:moment_estimate:6}, with $a = 1/d$ and $a = \frac{p-1}{p}$, there exists $C >0$ s.t. for all $t = 1,\dots,T-1$,
    \begin{equation} \label{eq:thm:moment_estimate:8}
            \E\big[\sum_{G\in \Phi^{N}_t}\mu^{N}(G)\WD(\mu_{G},\mu_{G}^{N}) \big] \leq CN^{\frac{t}{dT}-\frac{1}{d}} + CN^{\frac{t}{T}\frac{p-1}{p}-\frac{p-1}{p}} \add{+ C\Delta_N}
        \leq CN^{-\frac{1}{dT}} + CN^{-\frac{p-1}{pT}}.
    \end{equation}
    Combining \eqref{eq:thm:moment_estimate:1}, \eqref{eq:thm:moment_estimate:2} and \eqref{eq:thm:moment_estimate:8}, we conclude that there exists a constant $C > 0$ such that, for all $N \in \N$,
    \begin{equation*}
        \begin{split}
            \E\big[\AWD(\mu,\bar{\mu}^{N})\big] &\leq C\Delta_N + \E\big[C\WD(\mu_{1},\mu^{N}_1) + C\sum_{t=1}^{T-1}\sum_{G\in \Phi^{N}_t}\mu^{N}(G)\WD(\mu_{G},\mu_{G}^{N}) \big] \\ 
        &\leq CN^{-\frac{1}{dT}} + CN^{-\frac{q-1}{q}} + CN^{-\frac{p-1}{pT}}\leq CN^{-\frac{1}{dT}} + CN^{-\frac{p-1}{pT}},
        \end{split}
    \end{equation*}
    where we used the fact that  $\frac{q-1}{q} > \frac{p-1}{pT}$, which is \add{ satisfied when $q \geq \frac{T}{T-1}$, which holds }in both Setting \ref{setting:1} and Setting \ref{setting:2}.
\end{proof}

\section{Concentration inequality I}\label{sec:concentration1}
This section is devoted to the proof of Theorem \ref{thm:first_concentration}.
Similar to what done for the moment estimation in the previous section, we first prove a result, analogous to Lemma~\ref{lem:mean_shift_amp}, to address the uniform integrability issue.
\begin{lemma}\label{lem:deviation_shift_amp}
Let $\mu \in \mathcal{P}(\R^{dT})$. Assume Setting \ref{setting:1} or Setting \ref{setting:2}. Then, for all $\epsilon\in(0,p)$, there exists $C>0$ s.t., for all $t=1,\dots,T-1$, $j\geq 0$, $G \in \Phi^{N,j}_t$, $x>0$ and $N_G \in \N$,
    \begin{equation}\label{eq.conc_ui}
         \P\left[\WD(\mu_G,(\mu_G)^{N_G}) \geq  x\right]\leq C(\Vert G\Vert^r+1)^{p-\epsilon}(\mathrm{diam}(G)+1)^{p-\epsilon}N_G^{1-p+\epsilon}x^{-p + \epsilon}.
    \end{equation}
\end{lemma}
\begin{proof}
First notice that, by shifting and scaling of $\WD$, for all $t=1,\dots,T-1$, $G\in\Phi^N_t$, $f_G \in \R^{d}$, $\sigma_G, \lambda_G > 0$ and $N_G \in \N$,
    \begin{equation*}
        \WD(\mu_G,(\mu_G)^{N_G}) = \sigma_G\lambda_G\WD\bigg(\frac{1}{\lambda_G}\frac{\mu_G - f_G}{\sigma_G}, \Big(\frac{1}{\lambda_G}\frac{\mu_G - f_G}{\sigma_G}\Big)^{N_G}\bigg) .
    \end{equation*}
    Next, by Lemma~\ref{lem:shift_amp}, we know that $M_p^{1/p}\left(\frac{1}{\lambda_G}\frac{\mu_G - f_G}{\sigma_G}\right) \leq 1$. By Theorem~\ref{thm:concentration_wass_rate}-(i) below, we know that there exists $C>0$ such that, for all $x >0$ and $N_G\in\N$, 
    \begin{equation*}
    \begin{split}
    \P\Big[\WD(\mu_G,(\mu_G)^{N_G}) \geq  x\Big] &= \P\Big[\WD\bigg(\frac{1}{\lambda_G}\frac{\mu_G - f_G}{\sigma_G}, \Big(\frac{1}{\lambda_G}\frac{\mu_G - f_G}{\sigma_G}\Big)^{N_G}\bigg)\geq \frac{x}{\lambda_G\sigma_G}\Big] \leq C N_G^{1-p+\epsilon}\left(\frac{x}{\lambda_G\sigma_G}\right)^{-p + \epsilon}.
    \end{split}
    \end{equation*}
    Since $\sigma_G \leq c\Vert G \Vert^r +1$ and $\lambda_G = L_f \cdot \mathrm{diam}(G) + \big(\sup_{x_{1:t}\in \R^{dt}}M_{p}(\lambda_{x_{1:t}})\big)^{1/p}$, there exists $C>0$ s.t., for all $x >0$ and $N_G\in\N$, \eqref{eq.conc_ui} holds.
\end{proof}
Next, we prove a result to deal with deviation of sum of independent random variables. 
\begin{lemma}\label{lem:rosenthal_helper}
    Let $\beta_{+} > \beta \geq 2$, $\alpha_{i} \geq 0$ and let \add{$(X_i)_{i=1}^{\infty}$} be a sequence of independent nonnegative random variables such that for all $i \in \N$ there exists $\sigma_i > 0$ s.t., $\forall x > 0$, 
    \begin{equation*}
        \P[\vert X_i \vert \geq x] \leq \sigma_i x^{-\beta_{+}}.
    \end{equation*}
    Then there exists $C_{\beta,\beta_{+}} > 0$ s.t.
    \begin{equation*}
        \P\left[\vert S \vert \geq x\right] \leq \add{C_{\beta,\beta_{+}}\left( \sum_{i=1}^{\infty}\alpha_i^\beta \sigma_i^{(1+\beta_{+}-\beta)^{-1}} +   (\sum_{i=1}^{\infty}\E[\alpha_i X_i])^\beta\right)}x^{-\beta},
    \end{equation*}
    where \add{$S = \sum_{i=1}^{\infty}\alpha_i X_i$.}
\end{lemma}
\begin{proof}
First notice that, for all $i \in \N$,
\begin{align*}
    \E[\vert X_i\vert^\beta] = \int_{0}^{\infty}\beta t^{\beta-1}\P[\vert X_i \vert \geq t] \leq \varepsilon + \sigma_i\int_{\varepsilon}^{\infty}\beta t^{\beta-1 - \beta_{+}}dt = \varepsilon + \sigma_i\frac{\beta}{\beta_{+} - \beta}\varepsilon^{\beta - \beta_{+}}.
\end{align*}
By choosing $\varepsilon = \sigma_i^{r}, r = (1+\beta_{+}-\beta)^{-1}$, we have
\begin{equation*}
     \E[\vert X_i\vert^\beta] \leq \sigma_i^{r} + \frac{\beta}{\beta_{+}-\beta}\sigma_i^{1-(\beta_{+}-\beta)r} = \frac{\beta_{+}}{\beta_{+}-\beta}\sigma_i^{(1+\beta_{+}-\beta)^{-1}}.
\end{equation*}
Let $Y_i = \alpha_i X_i$ for all $i \in \N$, then
\begin{equation}\label{eq:lem:rosenthal_helper:1}
    \E[Y_i^\beta] \leq C_{\beta,\beta_{+}} \alpha_i^\beta \sigma_i^{(1+\beta_{+}-\beta)^{-1}},
\end{equation}
for some $C_{\beta,\beta_{+}}>0$.
Thus, by Rosenthal inequality II (see \cite{ibragimov2001best}), we have, for all $N \in \N$, 
\begin{align}\label{eq:lem:rosenthal_helper:2}
    \E\Big[\left( \sum_{i=1}^{N}Y_i\right)^{\beta}\Big]^{\frac{1}{\beta}} &\leq C \max\bigg\{ \big(\sum_{i=1}^{N}\E[ Y_i]\big), \big(\sum_{i=1}^{N}\E[Y_i^{\beta}]\big)^{\frac{1}{\beta}}\bigg\}.
\end{align}
On the other hand, \eqref{eq:lem:rosenthal_helper:1} yields
\begin{align}\label{eq7}
    \big(\sum_{i=1}^{N}\E[Y_i^{\beta}]\big)^{\frac{1}{\beta}} \leq C_{\beta,\beta_{+}} \big(\sum_{i=1}^{N}\alpha_i^\beta \add{\sigma_i^{(1+\beta_{+}-\beta)^{-1}}}\big)^{\frac{1}{\beta}}.
\end{align}
Now, combining \eqref{eq:lem:rosenthal_helper:2} and \eqref{eq7}, we have that 
\begin{align*}
    \E\Big[\left(\sum_{i=1}^{N}Y_i\right)^{\beta}\Big] &\leq \add{C\max\Big\{C_{\beta,\beta_{+}} \big(\sum_{i=1}^{N}\alpha_i^\beta \sigma_i^{(1+\beta_{+}-\beta)^{-1}}\big),  \Big(\sum_{i=1}^{N}\E[Y_i]\Big)^\beta\Big\}.}
\end{align*}
Then, by Markov inequality we conclude that, for all $x>0$,
\begin{equation*}
    \P\Big[ \sum_{i=1}^{N}\alpha_i X_i  \geq x\Big] \leq \E\Big[\left( \sum_{i=1}^{N}Y_i\right)^{\beta}\Big] x^{-\beta}\leq \add{C_{\beta,\beta_{+}}\left( \sum_{i=1}^{N}\alpha_i^\beta \sigma_i^{(1+\beta_{+}-\beta)^{-1}} +   \Big(\sum_{i=1}^{N}\E[\alpha_i X_i]\Big)^\beta\right)x^{-\beta}.}
\end{equation*}
By taking $N\to \infty$ we complete the proof. 
\end{proof}
We first recall the concentration rates for the classical Wasserstein distance, and then we prove our concentration rates for the adapted Wasserstein distance.
\begin{theorem}[Theorem 2 in \cite{fournier2015rate}]
    \label{thm:concentration_wass_rate}
    Let $\mu \in \mathcal{P}(\R^{d})$. 
    \begin{enumerate}[(i)]
        \item If $\exists q > 2$ s.t. $M_{q}(\mu) < \infty$, then, for all $\epsilon \in (0,q)$, there exists $C >0$ \add{(depending only on $q,\epsilon,M_q(\mu)$)} s.t., for all $N \in \N$ and $x > 0$,
        \begin{equation*}
            \P\big[\WD(\mu,\mu^{N}) \geq x\big] \leq CN(Nx)^{-q + \epsilon}.
        \end{equation*}
        \item If $\exists \alpha \geq 2, \gamma > 0$ s.t. $\mathcal{E}_{\alpha,\gamma}(\mu) < \infty$, then there exist $c, C >0$ s.t., for all $N \in \N$ and $x > 0$,
        \begin{align*}
            \P\big[\WD(\mu,\mu^{N}) \geq x\big] &\leq Ce^{-cNx^{d}}\mathbbm{1}_{\{x \leq 1\}} + Ce^{-cNx^{\alpha}}\mathbbm{1}_{\{x > 1\}}.
        \end{align*}    
    \end{enumerate}
\end{theorem}
\begin{proof}[Proof of Theorem \ref{thm:first_concentration}]
    We are going to estimate the concentration for each term in \eqref{eq:lem:helper:0} and \eqref{eq:lem:helper:1}.\\   
    \textbf{Step 1:} In this step, all statements hold for both uniform and non-uniform adapted empirical measures.
    First, we estimate the concentration rate of $\WD(\mu_1,\mu^{N}_1)$. Since by assumption $p > 2$, we can apply Theorem~\ref{thm:concentration_wass_rate}-(i) to $\mu_1$. Then, for all $\epsilon \in (0,p)$, there exists $C >0$ such that, for all $N \in \N$ and $x > 0$,
    \begin{equation}\label{eq:thm:concentration1:1}
        \P\Big[\WD(\mu_1,\mu_1^{N}) \geq  x\Big] \leq CN^{-p+1+\epsilon}x^{-p+\epsilon}.
    \end{equation}
    Next, we estimate the concentration rate of $\frac{\Delta_N}{N}\sum_{n=1}^{N}\Vert X^{(n)} \Vert$. Notice that 
    \begin{equation*}
        \begin{split}
            \Delta_N\frac{1}{N}\sum_{n=1}^{N}\Vert X^{(n)} \Vert &= \E[\Vert X^{(1)} \Vert]\Delta_N + \Delta_N\frac{1}{N}\sum_{n=1}^{N}\big(\Vert X^{(n)} \Vert - \E[\Vert X^{(n)} \Vert]\big)\\
            &\leq C\Delta_N + \Delta_N\frac{1}{N}\sum_{n=1}^{N}\big(\Vert X^{(n)} \Vert - \E[\Vert X^{(n)} \Vert]\big).
        \end{split}
    \end{equation*}
    For notational simplicity, we let $Z_{n} = \big(\Vert X^{(n)} \Vert - \E[\Vert X^{(n)} \Vert]\big)$ for all $n=1,\dots,N$. Then we have 
    \begin{equation}\label{eq:thm:concentration1:2}
        \begin{split}
            \P\Big[\Delta_N\frac{1}{N}\Big\vert\sum_{i=1}^{N}Z_i\Big\vert\geq x\Big] &= \P\Big[\Big\vert\sum_{i=1}^{N}Z_i\Big\vert \geq \frac{Nx}{\Delta_N}\Big] \leq \frac{\E\Big[\Big\vert\sum_{i=1}^{N}Z_i\Big\vert^p\Big]\Delta_N^p}{(Nx)^p}\qquad \text{(By Markov's inequality)}\\
            &\leq C\frac{\Big(N\E[\vert Z_1\vert^{p}] + N^{\frac{p}{2}}(\E[\vert Z_1 \vert^2])^{\frac{p}{2}}\Big)\Delta_N^p}{(Nx)^p}\qquad \text{(By Rosenthal inequality I, see \cite{ibragimov1998exact})}\\
            &\leq \frac{NC\Delta_N^p}{(Nx)^p} + \frac{N^{\frac{p}{2}}C\Delta_N^p}{(Nx)^p} \leq C N^{-\frac{p}{2}} x^{-p} \Delta_N^p.       
        \end{split}
    \end{equation}
    We are therefore left to estimate 
    \begin{equation*}
        Y_t \coloneqq \sum_{G\in \Phi_{t}^{N}}\mu^{N}(G)\WD(\mu_{G},\mu_{G}^{N}),
    \end{equation*}
    for all $t=1,\dots,T-1$. Combining Lemma~\ref{lem:helper2} and Lemma~\ref{lem:deviation_shift_amp}, we have that, for all $\epsilon \in (0,p)$, there exists $C > 0$ s.t., for all $N\in\N$, $x >0$ and $G \in \Phi_t^N$, 
    \begin{equation}\label{eq:thm:concentration1:3}
    \begin{split}
        \P\Big[ \WD(\mu_G,\mu^{N}_{G}) \geq  x \Big\vert \mu^{N}(G)\Big] &\leq  C(\Vert G\Vert^r+1)^{p-\epsilon}(\mathrm{diam}(G)+1)^{p-\epsilon}(N\mu^N(G))^{1-p+\epsilon}x^{-p + \epsilon}.
    \end{split}
    \end{equation}
    Let $\alpha_G = \mu^{N}(G) \geq 0$, $X_{G} = \WD(\mu_G,\mu_G^{N})\geq 0 $, $\sigma_G = CN_G^{-p+1+\epsilon}k_G^{p-\epsilon}$, $N_G = N\mu^N(G)$ and $k_G = (\Vert G\Vert^r+1)(\mathrm{diam}(G)+1)$, for all $G \in \Phi_t^N$.
    Then we can rewrite \eqref{eq:thm:concentration1:3} as
        \begin{align*}
            \P\big[ X_G \geq x \vert \mu^{N}(G)\big] \leq \sigma_G x^{-p+\epsilon}.
        \end{align*}
    Let $\delta \in (0, p-2-\epsilon)$, $\beta=p-\epsilon-\delta$ and $\beta_{+}= p-\epsilon$. Then we have $\beta_{+} > \beta > 2$, and we can apply Lemma~\ref{lem:rosenthal_helper}. Thus, there exists $C_{\beta_{+}, \beta} > 0$ s.t., for all $x > 0$ and $N \in \N$,
    \begin{equation}
    \label{eq:thm:concentration1:4}
        \P\Big[\sum_{G\in \Phi_{t}^{N}}\alpha_G X_{G} \geq x \Big\vert \mu^{N}(G)\Big] \leq C_{\beta_{+}, \beta} \sum_{G\in \Phi_{t}^{N}}\alpha_G^{\beta
        } \sigma_G^{(1+\beta_{+}-\beta)^{-1}}x^{-\beta} + C_{\beta_{+}, \beta}\Big(\sum_{G\in \Phi_{t}^{N}}\alpha_G \E[X_G|\alpha_G]\Big)^{\beta}x^{-\beta}.
    \end{equation}
    Now, we estimate the first term of the r.h.s. in \eqref{eq:thm:concentration1:4}. By plugging in the definition of $\sigma_G$, we have \begin{equation}\label{eq:thm:concentration1:5}
        \sum_{G\in \Phi_{t}^{N}}\alpha_G^{\beta
        } \sigma_G^{(1+\beta_{+}-\beta)^{-1}} = CN^{\frac{-p+1+\epsilon}{1 + \delta}}\sum_{G\in \Phi_{t}^{N}}\alpha_G^{\beta+\frac{-p+1+\epsilon}{1 + \delta}}k_G^{\frac{p-\epsilon}{1+\delta}}.
    \end{equation}
     Since $1-\delta < \frac{1}{1+\delta} <1$, we have that $\beta + \frac{-p+1+\epsilon}{1 + \delta}  > 1- (2+\epsilon-p)\delta > 1-(2+\epsilon)\delta$. Then, for $\delta \in (0,\frac{1}{2+\epsilon})$, by taking expectation of \eqref{eq:thm:concentration1:5} and using Jensen's inequality, we obtain that 
    \begin{equation}\label{eq:thm:concentration1:6}
    \begin{split}
        &\quad~\E\Big[\sum_{G\in \Phi_{t}^{N}}\alpha_G^{\beta
        } \sigma_G^{(1+\beta_{+}-\beta)^{-1}}\Big] = N^{\frac{-p+1+\epsilon}{1 + \delta}}\sum_{G\in \Phi_{t}^{N}}k_G^{\frac{p-\epsilon}{1+\delta}}\E\big[\alpha_G^{\beta+\frac{-p+1+\epsilon}{1 + \delta}}\big] \leq N^{\frac{-p+1+\epsilon}{1 + \delta}}\sum_{G\in \Phi_{t}^{N}}k_G^{\frac{p-\epsilon}{1 + \delta}}\E[\alpha_G^{1-(2+\epsilon)\delta}]\\
        &\leq N^{\frac{-p+1+\epsilon}{1 + \delta}}\sum_{G\in \Phi_{t}^{N}}k_G^{\frac{p-\epsilon}{1 + \delta}}\E[\alpha_G]^{1-(2+\epsilon)\delta} = N^{\frac{-p+1+\epsilon}{1 + \delta}}\sum_{G\in \Phi_{t}^{N}}k_G^{\frac{p-\epsilon}{1 + \delta}}\mu(G)^{1-(2+\epsilon)\delta}\quad (\alpha_G \leq 1).
    \end{split}
    \end{equation}
    Next, we estimate the second term of the r.h.s. in \eqref{eq:thm:concentration1:4}. By \eqref{eq:lem:mean_shift_amp:3} in Lemma~\ref{lem:mean_shift_amp} ($p>2 > \frac{d}{d-1}$), there exists $C>0$ s.t. 
    \begin{equation}
    \label{eq:thm:concentration1:7}
    \sum_{G\in \Phi_{t}^{N}}\alpha_G \E[X_G|\alpha_G]\leq C\sum_{G\in \Phi^{N}_t}\mu^{N}(G) (\Vert G\Vert^r +1)\left(\mathrm{diam}(G) + 1\right)(N\mu^N(G))^{-\frac{1}{d}} = CN^{-\frac{1}{d}}\sum_{G\in \Phi^{N}_t}\alpha_G^{1-\frac{1}{d}}k_G.
    \end{equation}
    By Lemma~\ref{lem:common_estimate} (with $\mu = \mu^N$, $a = 1/d$ and $b = c = 1$),
    \begin{equation}
    \label{eq:thm:concentration1:8}
    \left\{\begin{aligned}
        \sum_{G\in \hat{\Phi}^{N}_t}\alpha_G^{1-\frac{1}{d}}k_G &=  \sum_{G\in \hat{\Phi}^{N}_t}(\mu^N(G))^{1-\frac{1}{d}}(\Vert G\Vert^r+1)(\mathrm{diam}(G)+1) \leq CN^{\frac{t}{dT}}\sum_{j\geq 0}2^{j(r + t)}\mu^N(\calA_j)^{1-\frac{1}{d}}\\
        \sum_{G\in \check{\Phi}^{N}_t}\alpha_G^{1-\frac{1}{d}}k_G &= \sum_{G\in \check{\Phi}^{N}_t}(\mu^N(G))^{1-\frac{1}{d}}(\Vert G\Vert^r+1)(\mathrm{diam}(G)+1) \leq  C N^{\frac{t}{dT}}\sum_{j\geq 0}2^{j(r + 1)}\mu^N(\calA_j)^{1-\frac{1}{d}}.
    \end{aligned}\right.
    \end{equation}
    For simplicity, we use the notation $v(\hat{\Phi}^N_t,j)= 2^{j(r+t)}$ and $v(\check{\Phi}^N_t,j)= 2^{j(r+1)}$, and rewrite \eqref{eq:thm:concentration1:8} as 
    \begin{equation}
    \label{eq:thm:concentration1:9}
        \sum_{G\in \check{\Phi}^{N}_t}\alpha_G^{1-\frac{1}{d}}k_G \leq C N^{\frac{t}{dT}}\sum_{j\geq 0}v(\Phi^N_t,j)\mu^N(\calA_j)^{1-\frac{1}{d}}.
    \end{equation}
    Combining \eqref{eq:thm:concentration1:7} and \eqref{eq:thm:concentration1:9}, we have, for all $t=1,\dots,T-1$,
    \begin{equation}
    \label{eq:thm:concentration1:10}
    \begin{split}
        \E\Big[\Big(\sum_{G\in \Phi_{t}^{N}}\alpha_G \E[X_G|\alpha_G]\Big)^{\beta}\Big] \leq CN^{-\frac{\beta}{d}} \E\Big[\big(N^{\frac{t}{dT}}\sum_{j\geq 0}v(\Phi^N_t,j)\mu^N(\calA_j)^{1-\frac{1}{d}}\big)^{\beta}\Big]\leq CN^{-\frac{\beta}{dT}}\E\Big[\Big(\sum_{j\geq 0}v(\Phi^N_t,j)\mu^N(\calA_j)^{1-\frac{1}{d}}\Big)^{\beta}\Big].
    \end{split}
    \end{equation}
    By Minkowski's inequality,
    \begin{equation} 
    \label{eq:thm:concentration1:11}
    \begin{split}
        \E\Big[\Big(\sum_{j\geq 0}v(\Phi^N_t,j)\mu^N(\calA_j)^{1-\frac{1}{d}}\Big)^{\beta}\Big]^{\frac{1}{\beta}} &\leq \sum_{j\geq 0}\E\Big[\Big(v(\Phi^N_t,j)\mu^N(\calA_j)^{1-\frac{1}{d}}\Big)^{\beta}\Big]^{\frac{1}{\beta}}
        = \sum_{j\geq 0}\E\Big[v(\Phi^N_t,j)^{\beta}\mu^N(\calA_j)^{\beta(1-\frac{1}{d})}\Big]^{\frac{1}{\beta}}\\
        &\leq \sum_{j\geq 0}\E\Big[v(\Phi^N_t,j)^{\beta}\mu^N(\calA_j)\Big]^{\frac{1}{\beta}} = \sum_{j\geq 0}v(\Phi^N_t,j)\mu(\calA_j)^{\frac{1}{\beta}}, \quad (\beta > 2, \mu(\calA_j)\leq 1).
    \end{split}
    \end{equation}
    Combining the expectation of \eqref{eq:thm:concentration1:4}, \eqref{eq:thm:concentration1:6}, 
    \eqref{eq:thm:concentration1:10} and
    \eqref{eq:thm:concentration1:11}, we have 
    \begin{equation}
        \label{eq:thm:concentration1:12}
        \P\Big[\sum_{G\in \Phi_{t}^{N}}\alpha_G X_G \geq x \Big] \leq CN^{\frac{-p+1+\epsilon}{1 + \delta}}x^{-\beta}\sum_{G\in \Phi^{N}_t}k_G^{\frac{p-\epsilon}{1 + \delta}}\mu(G)^{1-(2+\epsilon)\delta} + CN^{-\frac{\beta}{dT}}x^{-\beta}\Big(\sum_{j\geq 0}v(\Phi^N_t,j)\mu(\calA_j)^{\frac{1}{\beta}}\Big)^{\beta}.
    \end{equation}
     \textbf{Step 2:} In this step, we prove that the infinite sums in both terms of the r.h.s. in \eqref{eq:thm:concentration1:12} are finite, by separately considering the uniform adapted empirical measure in Setting \ref{setting:1} and  the non-uniform adapted empirical measure in Setting \ref{setting:2}. \\ (i) Assume Setting \ref{setting:1}. Notice that $\mu(\mathcal{A}^{t}_{j})\leq \frac{M_{q}(\mu_{1:t})}{2^{q(j-1)}}$. For the first term of the r.h.s. in \eqref{eq:thm:concentration1:12}, by Lemma~\ref{lem:common_estimate} (with $a = (1+\epsilon)\delta$ and $b = c = \frac{p-\epsilon}{1+\delta}$), we have
    \begin{equation}
    \label{eq:thm:concentration1:13.1}
            \sum_{G\in \hat{\Phi}_{t}^{N}}k_G^{\frac{p-\epsilon}{1 + \delta}}\mu(G)^{1-(2+\epsilon)\delta} \leq CN^{\frac{t(2+\epsilon)\delta}{T}}\sum_{j\geq 0} 2^{j\big(r\frac{p-\epsilon}{1+\delta} + dt(2+\epsilon)\delta - q(1-(2+\epsilon)\delta)\big)} .
    \end{equation}
    For the second term of the r.h.s. in \eqref{eq:thm:concentration1:12}, we have
    \begin{equation}
    \label{eq:thm:concentration1:13.2}
            \sum_{j\geq 0}v(\hat{\Phi}^N_t,j)\mu(\calA_j)^{\frac{1}{\beta}} \leq C\sum_{j\geq 0} 2^{j(r+t - q/\beta)}.
    \end{equation}
    Since $q \geq rp + d(T-1)(p-1)$, there exists $\delta$ small enough such that the series in \eqref{eq:thm:concentration1:13.1} and  \eqref{eq:thm:concentration1:13.2} converge.\\
    (ii) Assume Setting \ref{setting:2}. For the first term of the r.h.s. in \eqref{eq:thm:concentration1:12}, by Lemma~\ref{lem:common_estimate} (with $a = (1+\epsilon)\delta$ and $b = c = \frac{p-\epsilon}{1+\delta}$), we have
    \begin{equation}
    \label{eq:thm:concentration1:14.1}
        \sum_{G\in \check{\Phi}^{N}_{t}}k_G^{\frac{p-\epsilon}{1 + \delta}}\mu(G)^{1-(2+\epsilon)\delta} \leq CN^{\frac{t(2+\epsilon)\delta}{T}}\sum_{j\geq 0} 2^{j\big((r+1)\frac{p-\epsilon}{1+\delta} - q(1-(2+\epsilon)\delta)\big)}.
    \end{equation}
    For the second term of the r.h.s. in \eqref{eq:thm:concentration1:12}, we have
    \begin{equation}
    \label{eq:thm:concentration1:14.2}
        \sum_{j\geq 0}v(\check{\Phi}^N_t,j)\mu(\calA_j)^{\frac{1}{\beta}} \leq C\sum_{j\geq 0} 2^{j(r+1 - q/\beta)}.
    \end{equation}
    Since $q \geq (r+1)p$, there exists $\delta$ small enough such that the series in \eqref{eq:thm:concentration1:14.1} and \eqref{eq:thm:concentration1:14.2} converge.\\
    \textbf{Step 3:} We choose $\delta$ small enough such that the series in \eqref{eq:thm:concentration1:13.1}, \eqref{eq:thm:concentration1:13.2}, \eqref{eq:thm:concentration1:14.1} and \eqref{eq:thm:concentration1:14.2} converge and 
    \begin{equation}
    \label{eq:thm:concentration1:15}
        \frac{-p+1+\epsilon}{1 + \delta} + (2+\epsilon)\delta \leq \frac{-p+\epsilon+\delta}{dT}, \quad -\beta = -p+\epsilon+\delta \leq -p + 2\epsilon.
    \end{equation}
    Combining this with \eqref{eq:thm:concentration1:12},
    \eqref{eq:thm:concentration1:13.1}, \eqref{eq:thm:concentration1:13.2}, \eqref{eq:thm:concentration1:14.1} and \eqref{eq:thm:concentration1:14.2}, we obtain that there exists $C > 0$ such that
    \begin{equation}
    \label{eq:thm:concentration1:16}
    \begin{split}
        \P\Big[\sum_{G\in \Phi_{t}^{N}}\alpha_G X_{G} \geq x \Big] \leq CN^{\frac{-p+1+\epsilon}{1 + \delta}}x^{-\beta}N^{\frac{t(2+\epsilon)\delta}{T}} + CN^{-\frac{\beta}{dT}}x^{-\beta} \leq  CN^{-\frac{p-2\epsilon}{dT}}x^{-p+2\epsilon}.
    \end{split}
    \end{equation}
    \textbf{Step 4:} (i) Assume Setting \ref{setting:1} and combine \eqref{eq:lem:helper:0}, \eqref{eq:thm:concentration1:1} and \eqref{eq:thm:concentration1:16}. Then we conclude that there exists  $C > 0$ such that, for all $N \in \N$ and $x > 0$,
    \begin{align*}
        \P\Big[\AWD(\mu,\hat{\mu}^{N}) \geq x + \mathrm{rate}(N)\Big] \leq  \add{CN^{-\frac{p-2\epsilon}{dT}}x^{-p+2\epsilon}}.
    \end{align*}
    (ii) Assume Setting \ref{setting:2} and combine \eqref{eq:lem:helper:1}, \eqref{eq:thm:concentration1:1}, \eqref{eq:thm:concentration1:2} and \eqref{eq:thm:concentration1:16}. Then we conclude that there exists  $C > 0$ such that, for all $N \in \N$ and $x > 0$,
    \begin{align*}
        \P\Big[\AWD(\mu,\check{\mu}^{N}) \geq x + \mathrm{rate}(N)\Big] \leq  \add{CN^{-\frac{p-2\epsilon}{dT}}x^{-p+2\epsilon}} + C N^{-\frac{p}{2}} x^{-p} .
    \end{align*}
    By the arbitrarity of $\epsilon$, we can replace $2\epsilon$ by $\epsilon$ and complete the proof of Theorem~\ref{thm:first_concentration}.
\end{proof}

\section{Concentration inequality II}\label{sec:concentration2} This section is devoted to the proof of Theorem \ref{thm:second_concentration}. 
Let us first recall the properties of sub-Gaussian distributions.
\begin{definition}[Sub-Gaussian distribution]
    A real-valued centered random variable $X$ is called sub-Gaussian if there exist $C > 0$ and $\sigma > 0$ s.t. 
    \begin{equation*}
        \P[| X | \geq x] \leq Ce^{-\frac{x^2}{2\sigma^2}}\quad \text{for all $x > 0$}.
    \end{equation*}
    In this case, we write $X \sim \mathrm{subG}(C,\sigma^2)$.
\end{definition}
\begin{lemma}\label{lem:subgaussian2}
    Let $X_{1},\dots,X_{n}$ be independent with $X_i \sim \mathrm{subG}(C,\sigma_i^2)$, and let $\alpha_i \geq 0$ and $\tilde{\sigma}_i = \sigma_i \max\{C,1\}$, for all $ i = 1,\dots,n$. Then we have $\sum_{i=1}^{n}\alpha_i X_i \sim \mathrm{subG}(2,4\sum_{i=1}^{n}\alpha_i^2 \tilde{\sigma_i}^2)$, that is,
    \begin{equation*}
        \P\Big[\big\vert\sum_{i=1}^{n}\alpha_i X_i\big\vert \geq x\Big] \leq 2e^{-\frac{x^2}{8\tilde{\sigma}^2}}\quad \text{for all $x > 0$},
    \end{equation*}
    where $\tilde{\sigma}^2 = \sum_{i=1}^{n}\alpha_i^2 \tilde{\sigma_i}^2$.
\end{lemma}
\begin{proof}
 See Section 2.5.2 in \cite{vershynin2018high}.
\end{proof}
\add{Similar to the estimation in Section~\ref{sec:concentration1}, here we will leverage the concentration inequalities of Wasserstein distance to estimate the essential quantities $\WD(\mu_{G},\mu_{G}^{N})$ in the proof of Theorem \ref{thm:second_concentration}.}
\begin{lemma}
\label{lem:wass_mean_concentration}
Assume $\mu \in \mathcal{P}(\R^{d})$. 
\begin{enumerate}[(i)]
    \item If $\exists q \in (1,2)$ s.t. $M_{q}(\mu) < \infty$, then there exists $C >0$ s.t., for all $N \in \N$ and $x > 0$,
    \begin{equation*}
        \P\Big[\Big\vert \WD(\mu,\mu^{N}) - \E\big[\WD(\mu,\mu^{N})\big] \Big\vert \geq  x\Big] \leq CN(Nx)^{-q}.
    \end{equation*}
    \item If $\exists q \geq 2$ s.t. $M_{q}(\mu) < \infty$, then there exists $C >0$ s.t., for all $N \in \N$ and $x > 0$, 
    \begin{equation*}
        \P\Big[\Big\vert \WD(\mu,\mu^{N}) - \E\big[\WD(\mu,\mu^{N})\big] \Big\vert \geq  x\Big] \leq CN^{-q/2}x^{-q}.
    \end{equation*}
    \item If $\exists \alpha \geq 2, \gamma > 0$ s.t. $\mathcal{E}_{\alpha,\gamma}(\mu) < \infty$, then there exist $c, C >0$ s.t., for all $N \in \N$ and $x > 0$,
    \begin{equation*}
        \P\Big[\Big\vert \WD(\mu,\mu^{N}) - \E\big[\WD(\mu,\mu^{N})\big] \Big\vert \geq  x\Big] \leq 2e^{-cNx^2}.
    \end{equation*}
\end{enumerate}
\end{lemma}
\begin{proof}
    See Section~4.1 in \cite{dedecker2015deviation}.
\end{proof}
Next, we prove a lemma similar to Lemma~\ref{lem:deviation_shift_amp} to address the uniform integrability issue.
\begin{lemma}\label{lem:exp_shift_amp}
Let $\mu \in \mathcal{P}(\R^{dT})$ and assume Setting \ref{setting:3}. Then there exists $c>0$ s.t., for all $t=1,\dots,T-1$, $j\geq 0$, $G \in \Phi^{N,j}_t$, $x>0$ and $N_G \in \N$,
    \begin{equation*}
    \begin{split}
    \P\left[\vert \WD(\mu_G,(\mu_G)^{N_G})- \E[\WD(\mu_G,(\mu_G)^{N_G})]\vert \geq  x\right] \leq 2\exp\left\{-\frac{cN_G x^2}{(1+\mathrm{diam}(G))^2(1+\Vert G \Vert^r)^2}\right\}.
    \end{split}
    \end{equation*}
\end{lemma}
\begin{proof}
    First, notice that by shifting and scaling of $\WD$, for all $t=1,\dots,T-1$, $G\in\Phi^N_t$, $f_G \in \R^{d}$, $\sigma_G, \lambda_G > 0$ and $N_G \in \N$,
    \begin{equation*}
        \WD(\mu_G,(\mu_G)^{N_G}) = \sigma_G \lambda_G\WD\left(\frac{1}{\lambda_G}\frac{\mu_G - f_G}{\sigma_G}, \left(\frac{1}{\lambda_G}\frac{\mu_G - f_G}{\sigma_G}\right)^{N_G}\right).
    \end{equation*}
    Next, by Lemma~\ref{lem:shift_amp}, we know that $\mathcal{E}_{1,\alpha}\left(\frac{1}{\lambda_G}\frac{\mu_G - f_G}{\sigma_G}\right) \leq e \big(\sup_{x_{1:t}\in \R^{dt}}\mathcal{E}_{\gamma,\alpha}(\lambda_{x_{1:t}})\big)$ where $\sigma_G = \sup_{x_{1:t}\in G}\sigma_t(x_{1:t}) \vee 1$ and $\lambda_G = 2L_f\mathrm{diam}(G) + 2\gamma^{-\alpha}$. Then, by Lemma~\ref{lem:wass_mean_concentration}-(iii), there exists $c>0$ s.t., for all $x >0$ and $N_G\in\N$, 
    \begin{equation*}
    \begin{split}
    &\quad~ \P\left[\vert \WD(\mu_G,(\mu_G)^{N_G})- \E[\WD(\mu_G,(\mu_G)^{N_G})]\vert \geq  x\right]\\
    &= \P\left[\left\vert S - \E[S]\right\vert\geq \frac{x}{\lambda_G\sigma_G}\right]\leq 2\exp\left\{-\frac{cN_G x^2}{\lambda_G^2\sigma_G^2}\right\}, \quad \left(S = \WD\left(\frac{1}{\lambda_G}\frac{\mu_G - f_G}{\sigma_G}, \left(\frac{1}{\lambda_G}\frac{\mu_G - f_G}{\sigma_G}\right)^{N_G}\right)\right).
    \end{split}
    \end{equation*}
     Since $\sigma_G \leq c\Vert G \Vert^r +1$ by Setting~\ref{setting:3}, there exists $c>0$ s.t., for all $x >0$ and $N_G\in\N$, 
    \begin{equation*}
    \begin{split}
    \P\left[\vert \WD(\mu_G,(\mu_G)^{N_G})- \E[\WD(\mu_G,(\mu_G)^{N_G})]\vert \geq  x\right] \leq 2\exp\left\{-\frac{cN_G x^2}{(1+\mathrm{diam}(G))^2(1+\Vert G \Vert^r)^2}\right\}.
    \end{split}
    \end{equation*}
\end{proof}
The two lemmas below summarize some estimations that will be used multiple times in the proof of Theorem \ref{thm:second_concentration}.
\begin{lemma}
\label{lem:compact_exp_bound}
Let $\alpha \geq 2$, $\gamma > 0$, and $\mu \in \mathcal{P}(\R^{dT})$ with finite $(\alpha,\gamma)$-exponential moment. Let $(X^{(n)})_{n\in\N}$ be i.i.d. samples of $\mu$. Then there exist $C,c>0$ s.t., for all $R\geq 1$, $N \in \N$,
\begin{equation*}
        \P[\max_{n=1,\dots,N}\Vert X^{(n)}\Vert > R] \leq CNe^{-cR^2}.
\end{equation*}
\end{lemma}
\begin{proof}
    Since $\mathcal{E}_{\gamma,\alpha}(\mu) < \infty$, for all $R\geq 1$ we have
    \begin{equation*}
        \P\left[\Vert X^{(1)}\Vert > R\right] = \P\left[\exp\{\gamma\Vert X^{(1)}\Vert^\alpha \}\geq \exp\{\gamma R^\alpha \}\right] \leq \mathcal{E}_{\gamma,\alpha}(\mu) e^{-\gamma R^\alpha} \leq \mathcal{E}_{\gamma,\alpha}(\mu) e^{-\gamma R^2}.
    \end{equation*}
    Thus, there exists $C>0$ s.t.  
    \begin{equation*}
        \begin{split}
            \P[\max_{n=1,\dots,N}\Vert X^{(n)}\Vert > R] &= 1 - (1 - \P[\Vert X^{(1)}\Vert > R])^{N}\leq 1 - (1 - Ce^{-\gamma R^2})^{N}\\
            &\leq N\vert 1-Ce^{-\gamma R^2}\vert Ce^{-\gamma R^2} + o(e^{-\gamma R^2})\quad (\text{Taylor expansion}).
        \end{split}
    \end{equation*}
    Therefore, there exist $C,c>0$ s.t., for all $R\geq 1$, $N \in \N$,
    \begin{equation*}
        \P[\max_{n=1,\dots,N}\Vert X^{(n)}\Vert > R] \leq CNe^{-cR^2}.
\end{equation*}
\end{proof}
\begin{lemma}
    \label{lem:mc1}
    Let $\mu\in\calP(\R^{dT})$ and $\mu^N$ be the empirical measure of $\mu$. Then, for all $t=1,\dots,T-1$, $N \in \N$ and $x > 0$,
    \begin{align*}
      \P\Big[\big\vert \tilde{Y}_t - \E[\tilde{Y}_t]\big\vert \geq x \Big] \leq 2e^{-2Nx^2},
    \end{align*}
    where 
    \begin{equation*}
        \tilde{Y}_t = N^{-1}\sum_{G\in \Phi^{N}_t}(N\mu^{N}(G))^{1-\frac{1}{d}}.
    \end{equation*}
    
\end{lemma}
\begin{proof}
    Let $\tilde{Y}_t = N^{-1}\sum_{G\in \Phi^{N}_t}(N\mu^{N}(G))^{1-\frac{1}{d}}$ and $f(x_1,\dots,x_N) = \sum_{G\in \Phi^{N}_t} \Big(\sum_{j=1}^{N}\mathbbm{1}_{G}(x_j)\Big)^{1-\frac{1}{d}}$. Notice that $\tilde{Y}_{t} = \frac{1}{N}f(X_1,\dots, X_N)$. We are going to show that $f$ satisfies the $(c_1,\dots,c_N)$-bounded differences property in \add{McDiarmid's inequality, see \cite{mcdiarmid1989method}}. To simplify notations, for all $G \in \Phi^N_t$ and $x \in (\R^{dt})^N$, we set $f_{G}(x) = \Big(\sum_{j=1}^{N}\mathbbm{1}_{G}(x_j)\Big)^{1-\frac{1}{d}}$. \add{Fix any $i = 1,\dots, N$ and let $x = (x_1, \dots, x_N)$ and $ x' = (x'_1,\dots, x'_N) \in \R^{dTN}$ differ only in the $i$-th coordinate ($x_j = x_j^{\prime}$ for all $j \neq i$).} Then $f(x) - f(x') = \sum_{G \in \Phi_t^N}\Big(f_{G}(x) - f_{G}(x')\Big)$. Now, if there exists $G^{*}\in\Phi^N_t$ such that $x_i,x'_i \in G^{*}$, then clearly $f(x) = f(x')$. Thus, it remains to consider the case when there exist disjoint $G^{*}, G^{*'}\in\Phi^N_t$ such that $x_i\in G^{*}$ while $x'_i\in G^{*'}$. In this case, for all $G \in \Phi^N_t-\{G^{*},G^{*'}\}$, $x_i \not\in G$ and $x'_i \not \in G$, we have
    \begin{equation*}
        f_{G}(x) = \Big(\sum_{j\neq i}\mathbbm{1}_{G}(x_j)\Big)^{1-\frac{1}{d}} = \Big(\sum_{j\neq i}\mathbbm{1}_{G}(x'_j)\Big)^{1-\frac{1}{d}} = f_{G}(x').
    \end{equation*}
    Therefore 
    \begin{equation*}
        \begin{aligned}
            f(x) - f(x') &= \sum_{G \in \Phi_t^N}\big(f_{G}(x) - f_{G}(x')\big)
            = \sum_{G \in \{G^{*},G^{*'}\}}\Big(f_{G}(x) - f_{G}(x')\Big)\\
            &= \sum_{G \in \{G^{*},G^{*'}\}}\Big[\Big(\sum_{j=1}^{N}\mathbbm{1}_{G}(x_j)\Big)^{1-\frac{1}{d}} - \Big(\sum_{j=1}^{N}\mathbbm{1}_{G}(x'_j)\Big)^{1-\frac{1}{d}}\Big]\\
            &= \sum_{G \in \{G^{*},G^{*'}\}}\Big[\Big(\sum_{j\neq i}\mathbbm{1}_{G}(x_j) + \mathbbm{1}_{G}(x_i)\Big)^{1-\frac{1}{d}} - \Big(\sum_{j\neq i}\mathbbm{1}_{G}(x'_j) + \mathbbm{1}_{G}(x'_i)\Big)^{1-\frac{1}{d}}\Big]\\
            &= \sum_{G \in \{G^{*},G^{*'}\}}\Big[\Big(\sum_{j\neq i}\mathbbm{1}_{G}(x_j) + \mathbbm{1}_{G}(x_i)\Big)^{1-\frac{1}{d}} -\Big(\sum_{j\neq i}\mathbbm{1}_{G}(x_j) + \mathbbm{1}_{G}(x'_i)\Big)^{1-\frac{1}{d}}\Big].
        \end{aligned}
    \end{equation*}
    To simplify notation, we let $z = \sum_{j\neq i}\mathbbm{1}_{G^{*}}(x_j)$ and $ z' = \sum_{j\neq i}\mathbbm{1}_{G^{*'}}(x_j)$. Then
    \begin{equation*}
       \begin{aligned}
        f(x) - f(x') &= \Big[\big(z + 1\big)^{1-\frac{1}{d}} - (z)^{1-\frac{1}{d}}\Big] + \Big[(z')^{1-\frac{1}{d}} - \big(z' + 1\big)^{1-\frac{1}{d}}\Big]\\
        &\leq \max_{z_{0}\geq 0}\Big[\big(z_{0} + 1\big)^{1-\frac{1}{d}} - (z_{0})^{1-\frac{1}{d}}\Big]= 1.
    \end{aligned}
    \end{equation*}
    Clearly, the same estimate holds when exchanging the role of $x$ and $x'$, thus we proved the claim that $f$ satisfies the $(c_1,\dots,c_N)$-bounded differences property with $c_i = 1, \forall i\in\N$. Therefore, we can apply \add{McDiarmid's inequality (see \cite{mcdiarmid1989method})} with $c_i = 1, \forall i\in\N$, which implies that there exist $C,c >0$ such that, for all $N \in \N$ and $x > 0$,
    \begin{equation*}
        \P\Big[\big\vert \tilde{Y}_t - \E[\tilde{Y}_t] \big\vert \geq x \Big] \leq 2e^{-2 N x^2} .
    \end{equation*}
\end{proof}

\begin{proof}[Proof of Theorem \ref{thm:second_concentration}]
    \hypertarget{proof:second_concentration_uniform}{\,\!} 
    \textbf{Step 1:} In this step, all statements hold for both uniform and non-uniform adapted empirical measures. Note that, by Lemma \ref{lem:helper}, there exists $C > 0$ such that, for all $N \in \N$,
    \begin{equation}\label{eq:thm:concentration2:1}
            \AWD(\mu,\bar{\mu}^{N}) \leq C\Big(\Delta_N  + \WD(\mu_{1},\mu^{N}_{1}) + \sum_{t=1}^{T-1}\sum_{G\in\Phi_{t}^{N}}\mu^{N}(G)\WD(\mu_{G},\mu_{G}^{N})\Big)
            + \frac{C\Delta_N}{N}\sum_{n=1}^{N}\Vert X^{(n)} \Vert.
    \end{equation}
    We are now going to estimate the deviation of each term. First, we estimate the deviation of $\WD(\mu_1,\mu^{N}_1)$. Since in Setting~\ref{setting:3} we have $\alpha \geq 2$, we can apply Lemma~\ref{lem:wass_mean_concentration}-(iii) to $\mu_1$. Then, there exists $C >0$ s.t., for all $N \in \N$ and $x > 0$,
    \begin{equation}\label{eq:thm:concentration2:2}
        \P\Big[\Big\vert \WD(\mu_1,\mu_1^{N}) - \E\big[\WD(\mu_1,\mu_1^{N})\big] \Big\vert \geq  x\Big] \leq 2e^{-cNx^2}.
    \end{equation}
    Moreover, since finite $(\alpha,\gamma)$-exponential moment with $\alpha \geq 2$ implies finite moments of any order, we can apply Theorem~\ref{thm:wass_rate} to $\mu_1$ and obtain, for $q \geq 1 + \frac{1}{d-1}$, that
    \begin{equation}\label{eq:thm:concentration2:3}
        \E\big[\WD(\mu_{1},\mu^{N}_{1})\big] \leq CN^{-\frac{1}{d}} + CN^{-\frac{q-1}{q}} \leq CN^{-\frac{1}{d}} \leq \mathrm{rate}(N).
    \end{equation}
    Thus, combining \eqref{eq:thm:concentration2:2} and \eqref{eq:thm:concentration2:3}, we have that 
    \begin{equation}\label{eq:thm:concentration2:4}
        \P\big[\WD(\mu_{1},\mu^{N}_{1}) \geq x + \mathrm{rate}(N)\big] \leq 2e^{-cNx^2}.
    \end{equation}
    Next, we estimate the deviation of $\frac{\Delta_N}{N}\sum_{n=1}^{N}\Vert X^{(n)} \Vert$. Notice that 
    \begin{equation}
    \label{eq:thm:concentration2:5}
        \begin{split}
            \Delta_N\frac{1}{N}\sum_{n=1}^{N}\Vert X^{(n)} \Vert \leq C\Delta_N + \Delta_N\frac{1}{N}\sum_{n=1}^{N}\big(\Vert X^{(n)} \Vert - \E[\Vert X^{(n)} \Vert]\big).
        \end{split}
    \end{equation}
    As before, we let $Z_{n} = \big(\Vert X^{(n)} \Vert - \E[\Vert X^{(n)} \Vert]\big)$ for all $n=1,\dots,N$. Notice that $Z_{1},\dots,Z_{N}$ are i.i.d. and $Z_1\sim \mathrm{subG}(1,\sigma^2)$, $\sigma^2 = \frac{1}{2\gamma}$, because of the exponential moment assumption. Then, by Lemma~\ref{lem:subgaussian2} we have $\frac{1}{N}\sum_{n=1}^{N}Z_n \sim \mathrm{subG}(2,\frac{4\sigma^2}{N})$, i.e., for all $x >0$ and $N \in \N$,
    \begin{equation}\label{eq:thm:concentration2:6}
        \P\Big[\frac{1}{N}\sum_{n=1}^{N}Z_n\geq x\Big] \leq 2e^{-\frac{Nx^2}{8\sigma^2}}.
    \end{equation}
    Combining \eqref{eq:thm:concentration2:5} and \eqref{eq:thm:concentration2:6}, we conclude that 
    \begin{equation}\label{eq:thm:concentration2:7}
    \P\Big[\Delta_N\frac{1}{N}\sum_{n=1}^{N}\Vert X^{(n)} \Vert \geq x + \mathrm{rate}(N)\Big] \leq 2e^{-\frac{Nx^2}{8\sigma^2}}.
    \end{equation}
    Now let $Y_t \coloneqq\sum_{G\in\Phi_{t}^{N}}\mu^{N}(G)\WD(\mu_{G},\mu_{G}^{N})$, for $t = 1,\dots, T-1$. Combining \eqref{eq:thm:concentration2:1}, \eqref{eq:thm:concentration2:4} and \eqref{eq:thm:concentration2:7}, we have that there exist $C,c >0$ s.t., for all $x>0$,
     \begin{equation}
     \label{eq:thm:concentration2:8}
        \begin{split}
            \P\Big[\AWD(\mu,\bar{\mu}^{N}) \geq x + \mathrm{rate}(N)\Big]\leq \sum_{t=1}^{T-1}\P\Big[Y_t \geq \frac{x}{3T} + \mathrm{rate}(N)\Big] + Ce^{-cNx^2}.
        \end{split}
    \end{equation}
    Therefore, we are left to estimate the deviation of $Y_t$. By Lemma~\ref{lem:helper2} and Lemma~\ref{lem:exp_shift_amp}, there exists $c>0$ s.t., for all $t=1,\dots,T-1$, $j\geq 0$, $G \in \Phi^{N,j}_t$, $x>0$ and $N \in \N$,
    \begin{equation}
    \label{eq:thm:concentration2:9}
    \begin{split}
    \P\big[|\WD(\mu_G,\mu_G^{N})- \E[\WD(\mu_G,\mu_G^{N})\vert\mu^N(G)]| \geq x \big\vert \mu^{N}(G) \big] \leq 2\exp\left\{-\frac{cN_G x^2}{(1+\mathrm{diam}(G))^2(1+\Vert G \Vert^r)^2}\right\}.
    \end{split}
    \end{equation}
    Let $\alpha_G = \mu^{N}(G) \geq 0$, $X_{G} = \WD(\mu_G,\mu_G^{N}) - \add{\E[\WD(\mu_G,\mu_G^{N})|\mu^N(G)]}$ and $\sigma_G^2 = \add{\frac{(1+\mathrm{diam}(G))^2(1+\Vert G \Vert^r)^2}{2c(N\mu^{N}(G))}}$, for all $G \in \Phi_t^N$. Then, for all $G \in \Phi_t^N$, $X_G$ is centered by definition, and \add{$(X_G)_{G\in\Phi^N_t}$ are independent given $\{\mu^{N}(G)\colon G\in\Phi_{t}^{N}\}$ by Lemma~\ref{lem:helper2}.} Moreover, by rewriting \eqref{eq:thm:concentration2:9} as
    \begin{equation*}
        \P\big[|X_G| \geq x\big\vert \mu^{N}(G)\big] \leq C e^{-\frac{x^2}{2\sigma^2_G}},
    \end{equation*}
    we have that $X_G \sim \mathrm{subG}(C,\sigma_G^2)$ conditional on $\{\mu^{N}(G)\colon G\in\Phi_{t}^{N}\}$. Thus we can apply Lemma \ref{lem:subgaussian2} and deduce that $\sum_{G \in \Phi_t^N}\alpha_G X_G \sim \mathrm{subG}(2,4\tilde{\sigma}^2\max\{1,C^2\})$ conditional on $\{\mu^{N}(G)\colon G\in\Phi_{t}^{N}\}$, i.e., for all $x>0$,
    \begin{equation*}
        \mathbb{P}\Big[\Big\vert\sum_{G \in \Phi_t^N}\alpha_G X_G\Big\vert \geq x \Big\vert \mu^N(G) \Big]  \leq 2\exp\Big\{-\frac{x^2}{8\tilde{\sigma}^2\max\{1,C^2\}}\Big\},
    \end{equation*}
    where
    \begin{align*}
        \tilde{\sigma}^2 &= \sum_{G \in \Phi_t^N}\alpha_G^2 \sigma_G^2
        = \sum_{G \in \Phi_t^N}\mu^{N}(G)^2 \frac{\add{(1+\mathrm{diam}(G))^2}(1+\Vert G \Vert^r)^2}{2cN\mu^{N}(G)}
        = \frac{1}{2cN}\add{\sum_{G \in \Phi_t^N}\mu^{N}(G)(1+\mathrm{diam}(G))^2(1+\Vert G \Vert^r)^2.}
    \end{align*}
    Therefore, we conclude that there exist $C, c > 0$ s.t., for all $x>0$,
    \begin{equation}\label{eq:thm:concentration2:10}
    \begin{aligned}
        &\quad~\P\Big[\Big\vert Y_t - \E\big[Y_t\vert\mu^N(G)\big]\Big\vert \geq x\Big] = \P\Big[\Big\vert \sum_{G\in\Phi_{t}^{N}}\mu^{N}(G)\big(\WD(\mu_{G},\mu_{G}^{N}) - \E\big[\WD(\mu_{G},\mu_{G}^{N})\vert\,\mu^{N}(G)\big]\big)\Big\vert \geq x\Big]\\
        &= \P\Big[\big\vert\sum_{G \in \Phi_t^N}\alpha_G X_G\big\vert \geq x\Big]
        = \E\Big[\P\Big[\big\vert\sum_{G \in \Phi_t^N}\alpha_G X_G\big\vert \geq x  \Big\vert\,\mu^{N}(G)\Big] \Big]\leq\E\Big[2\exp\Big\{-\frac{x^2}{8\tilde{\sigma}^2\max\{1,C^2\}}\Big\} \Big]\\ 
        &= \add{\E\Big[2\exp\Big\{-\frac{cNx^2}{4\max\{1,C^2\}\sum_{G \in \Phi_t^N}\mu^{N}(G)(1+\mathrm{diam}(G))^2(1+\Vert G \Vert^r)^2}\Big\}\Big].}
    \end{aligned}
    \end{equation}
    Notice that $Y_t \leq \vert Y_t - \E[Y_t\vert\,\mu^{N}(G)]\vert + \E[Y_t\vert\,\mu^{N}(G)]$. In the next three steps, we estimate the deviation of $\vert Y_t - \E[Y_t\vert\,\mu^{N}(G)]\vert$ and $\E[Y_t\vert\,\mu^{N}(G)]$.\\
    \textbf{Step 2:} In this step, we estimate the deviation of $\vert Y_t - \E[Y_t\vert\,\mu^{N}(G)]\vert$ by considering separately three cases: (i) uniform grid and $r=0$; (ii) uniform grid and $r>0$; (iii) non-uniform grid and $r\geq 0$.\\
    (i) Under uniform grid and $r=0$, $\sum_{G \in \hat{\Phi}_t^N}\mu^{N}(G)(1+\mathrm{diam}(G))^2(1+\Vert G \Vert^r)^2 \leq 8$. Then, from \eqref{eq:thm:concentration2:10}, there exists $c>0$ s.t., for all $x>0$ and $N\in\N$,
    \begin{equation}
    \label{eq:thm:concentration2:11}
        \P\Big[\Big\vert Y_t - \E\big[Y_t\vert\,\mu^{N}(G)\big]\Big\vert \geq x\Big]  \leq 2\exp\Big\{-cNx^2\Big\}.
    \end{equation}
    (ii) Under uniform grid and $r>0$, by Lemma \ref{lem:compact_exp_bound}, there exist $C,c > 0$ s.t., for all $R\geq 1$,
    \begin{equation}
    \label{eq:thm:concentration2:13}
        \P[\max_{n=1,\dots,N}\Vert X^{(n)}\Vert > R] \leq CNe^{-cR^2}.
    \end{equation}
    Notice that, for any $R\geq 1$, if $\Vert X^{(n)} \Vert \leq R$ for all $n=1,\dots,N$, then $\sum_{G \in \hat{\Phi}_t^N}\mu^{N}(G)(1+\mathrm{diam}(G))^2(1+\Vert G \Vert^r)^2 \leq 4(1+R^r)^2 \leq 16R^{2r}$. Combining this with \eqref{eq:thm:concentration2:10} and \eqref{eq:thm:concentration2:13}, we obtain that there exist $C, c_1,c_2 > 0$ s.t., for all $R\geq 1$,
    \begin{equation}
    \label{eq:thm:concentration2:14}
        \P\Big[\Big\vert Y_t - \E\big[Y_t\vert\,\mu^{N}(G)\big]\Big\vert \geq x\Big] \leq 2\exp\Big\{-\frac{c_1 Nx^2}{16R^{2r}}\Big\} + CNe^{-c_2 R^2}.
    \end{equation}
    By choosing $R = (\sqrt{N}x)^{\frac{1}{r+1}}$, there exist $c,C>0$ s.t., for all $x>0$ and $N \in \N$ satisfying $\sqrt{N}x \geq 1$,
    \begin{equation}
    \label{eq:thm:concentration2:15}
        \P\Big[\Big\vert Y_t - \E\big[Y_t\vert\,\mu^{N}(G)\big]\Big\vert \geq x\Big] \leq CNe^{-c (\sqrt{N}x)^{\frac{1}{r+1}}}.
    \end{equation}
    (iii) Under non-uniform grid and $r\geq 0$, notice that, for any $R\geq 1$, if $\Vert X^{(n)} \Vert \leq R$ for all $n=1,\dots,N$, then $\sum_{G \in \check{\Phi}_t^N}\mu^{N}(G)(1+\mathrm{diam}(G))^2(1+\Vert G \Vert^r)^2 \leq (1+R)^2(1+R^r)^2 \leq 16 R^{2(r+1)}$. Thus, by the same argument used in (ii), we conclude that there exist $c,C>0$ s.t., for all $x>0$ and $N \in \N$ satisfying $\sqrt{N}x \geq 1$,
    \begin{equation}
    \label{eq:thm:concentration2:15.1}
        \P\Big[\Big\vert Y_t - \E\big[Y_t\vert\,\mu^{N}(G)\big]\Big\vert \geq x\Big] \leq CNe^{-c (\sqrt{N}x)^{\frac{1}{r+2}}}.
    \end{equation}
    \textbf{Step 3:} In this step, we start estimating the deviation of $\E[Y_t\vert\,\mu^{N}(G)]$. All statements here hold for both uniform and non-uniform cases. Since finite $(\alpha,\gamma)$-exponential moment with $\alpha\geq 2$ implies finite $p$-exponential moment for all $p\geq 1$, Setting~\ref{setting:3} implies Setting~\ref{setting:1} and Setting~\ref{setting:2}. Then we can invoke \eqref{eq:lem:mean_shift_amp:3} in the proof of Lemma~\ref{lem:mean_shift_amp} with $p > \frac{d}{d-1}$, and  conclude that there exists $C>0$ s.t.  
    \begin{equation}\label{eq:thm:concentration2:16}
        \E\big[Y_t|\mu^{N}(G) \big] \leq CN^{-1}\sum_{G\in \Phi^{N}_t}\add{(\Vert G\Vert^r+1)(\mathrm{diam}(G) + 1)}(N\mu^{N}(G))^{1-\frac{1}{d}}.
    \end{equation}
    For simplicity, we use the notation
    \begin{equation*}
        \tilde{Y}_{t} \coloneqq N^{-1}\sum_{G\in \Phi^{N}_t}(N\mu^{N}(G))^{1-\frac{1}{d}}, \quad \add{\hat{Y}_t  \coloneqq N^{-1}\sum_{G\in \Phi^{N}_t}(\Vert G\Vert^r+1)(\mathrm{diam}(G) + 1)(N\mu^{N}(G))^{1-\frac{1}{d}}. }
    \end{equation*}
    \add{Then, \eqref{eq:thm:concentration2:16} is equivalent to $\E\big[Y_t|\mu^{N}(G) \big] \leq C \hat{Y}_t$. Here $\tilde{Y}_t$ is introduced in order to bound $\hat{Y}_t$ in the next step, so we estimate $\tilde{Y}_t$ first.} By Lemma~\ref{lem:mc1}, for all $t=1,\dots,T-1$, $N \in \N$ and $x > 0$,
    \begin{equation}\label{eq:thm:concentration2:17}
        \P\Big(\big\vert \tilde{Y}_t - \E[\tilde{Y}_t] \big\vert \geq x \Big) \leq 2e^{-2 N x^2}.
    \end{equation}
    Recall \eqref{eq:thm:moment_estimate:8} in the proof of Theorem \ref{thm:moment_estimate}, which shows that $\E[\tilde{Y}_{t}] \leq \mathrm{rate}(N)$. Combining this and \eqref{eq:thm:concentration2:17}, we conclude that, for all $N\in\N$ and $x>0$,
    \begin{align}\label{eq:thm:concentration2:19}
        \P\Big[\tilde{Y}_t \geq x + \mathrm{rate}(N)\Big] \leq \P\Big[\Big\vert \tilde{Y}_t - \E[\tilde{Y}_t] \Big\vert  \geq x \Big] + \P\Big[\E[\tilde{Y}_t]\geq \mathrm{rate}(N)\Big]\leq 2e^{-2Nx^2},
    \end{align}
    \add{where we recall that $\mathrm{rate}(N) = C \Delta_N = CN^{-\frac{1}{dT}}$ and $C$ is a generic constant.}\\
    \textbf{Step 4:} In this step, we finish estimating the deviation of $\E[Y_t\vert\,\mu^{N}(G)]$ and complete the proof. We bound the deviation of $\hat{Y}_t$ by considering separately three cases: (i) uniform grid and $r=0$; (ii) uniform grid and $r>0$; (iii) non-uniform grid and $r\geq 0$.\\
    (i) Under uniform grid and $r=0$, for all $G\in\hat{\Phi}^N_t$ we have that $\mathrm{diam}(G) \leq 1$ and $\Vert G \Vert^r \leq 1$, thus $\hat{Y}_t \leq 4 \tilde{Y}_t$. Therefore, by \eqref{eq:thm:concentration2:19}, there exist $c,C>0$ s.t., for all $x>0$, 
    \begin{equation}
    \label{eq:thm:concentration2:20}
        \begin{split}
            \P\Big[\hat{Y}_t \geq x + \mathrm{rate}(N)\Big] &\leq \P\Big[\tilde{Y}_t \geq x/4 + \mathrm{rate}(N)\Big]  \leq 2e^{-cNx^2}.
        \end{split}
    \end{equation}
    Combining \eqref{eq:thm:concentration2:11}, 
    \eqref{eq:thm:concentration2:16} and \eqref{eq:thm:concentration2:20}, we obtain that there exist $C,c >0$ s.t., for all $x>0$, $N \in \N$,
    \begin{equation*}
        \begin{split}
            \P\Big[Y_t \geq x + \mathrm{rate}(N)\Big] &\leq C\P\Big[\Big\vert Y_t - \E\big[Y_t\vert\,\mu^{N}(G)\big]\Big\vert \geq x/2 + \mathrm{rate}(N)\Big] + C\P\Big[ \E\big[Y_t\vert\,\mu^{N}(G)\big] \geq x/2 + \mathrm{rate}(N)\Big] \\
            &\leq Ce^{-cNx^2} + C\P\Big[ C\hat{Y}_t \geq x + \mathrm{rate}(N)\Big]  \leq Ce^{-cNx^2}.
        \end{split}
    \end{equation*}
    Then, by \eqref{eq:thm:concentration2:8}, there exist $C,c >0$ s.t., for all $x>0$ and $N \in \N$ satisfying $N\sqrt{x} \geq 1$,
    \begin{equation*}
        \begin{split}
            \P\Big[\AWD(\mu,\hat{\mu}^{N}) \geq x + \mathrm{rate}(N)\Big] \leq Ce^{-cNx^2}.
        \end{split}
    \end{equation*}
    (ii) Under uniform grid and $r>0$, for any $R\geq 1$, if $\Vert X^{(n)}\Vert \leq R$ for all $n \in \N$, then $\hat{Y}_t \leq (1+R^r) \tilde{Y}_t \leq 4R^r \tilde{Y}_t$. Combining this and \eqref{eq:thm:concentration2:19} with $R = (\sqrt{N}x)^{\frac{1}{r+1}}$, we conclude that there exist $C,c >0$ s.t., for all $x>0$ and $N \in \N$ satisfying $N\sqrt{x} \geq 1$,
    \begin{equation*}
            \begin{split}
                \P\Big[\hat{Y}_t \geq x\Big] \leq \P\Big[\tilde{Y}_t \geq \frac{x}{4R^r}\Big] +  \P[\max_{n=1,\dots,N}\Vert X^{(n)}\Vert > R] \leq 2e^{\frac{-cNx^2}{16R^{2r}}} + CNe^{-cR^2} \leq CNe^{-c (\sqrt{N}x)^{\frac{1}{r+1}}}.
            \end{split}
    \end{equation*}
    Together with
    \eqref{eq:thm:concentration2:8}, \eqref{eq:thm:concentration2:15}, 
    \eqref{eq:thm:concentration2:16} and \eqref{eq:thm:concentration2:20}, this ensures the existence of $C,c >0$ s.t., for all $x>0$ and $N \in \N$ satisfying $N\sqrt{x} \geq 1$,
    \begin{equation*}
        \begin{split}
            \P\Big[\AWD(\mu,\hat{\mu}^{N}) \geq x + \mathrm{rate}(N)\Big] \leq CNe^{-c (\sqrt{N}x)^{\frac{1}{r+1}}}.
        \end{split}
    \end{equation*}
    (iii) Under non-uniform grid and $r\geq 0$, for any $R\geq 1$, if $\Vert X^{(n)}\Vert \leq R$ for all $n \in \N$, then $\hat{Y}_t \leq (1+R)(1+R^r) \tilde{Y}_t \leq 4R^{r+1}\tilde{Y}_t$. Combining this and \eqref{eq:thm:concentration2:19} with  $R = (\sqrt{N}x)^{\frac{1}{r+2}}$, we obtain that there exist $C,c >0$ s.t., for all $x>0$ and $N \in \N$ satisfying $N\sqrt{x} \geq 1$,
    \begin{equation*}
            \begin{split}
                \P\Big[\hat{Y}_t \geq x\Big] \leq \P\Big[\tilde{Y}_t \geq \frac{x}{4R^{r+1}}\Big] +  \P[\max_{n=1,\dots,N}\Vert X^{(n)}\Vert > R] \leq 2e^{\frac{-cNx^2}{16R^{2(r+1)}}} + CNe^{-cR^2} \leq CNe^{-c (\sqrt{N}x)^{\frac{1}{r+2}}}.
            \end{split}
    \end{equation*}
    Together with
    \eqref{eq:thm:concentration2:8}, \eqref{eq:thm:concentration2:15.1}, 
    \eqref{eq:thm:concentration2:16} and \eqref{eq:thm:concentration2:20}, this ensures the existence of $C,c >0$ s.t., for all $x>0$ and $N \in \N$ satisfying $N\sqrt{x} \geq 1$,
    \begin{equation*}
        \begin{split}
            \P\Big[\AWD(\mu,\hat{\mu}^{N}) \geq x + \mathrm{rate}(N)\Big] \leq CNe^{-c (\sqrt{N}x)^{\frac{1}{r+2}}}.
        \end{split}
    \end{equation*}
\end{proof}
\section{Almost sure convergence}\label{sec:asconvergence} This section is dedicated to the proof of Theorem \ref{thm:asconvergence}. We start by considering the compact case.
\begin{lemma}
    \label{lem:asconvergence_compact}
    Let $\mu \in \mathcal{P}(\R^{dT})$ be compactly supported and let $\bar{\mu}^{N}$ be the uniform or non-uniform adapted empirical measure. Then
    \begin{equation*}
         \lim_{N \to \infty} \AWD(\mu,\bar{\mu}^{N}) = 0\quad \P\text{-a.s.}
    \end{equation*}
\end{lemma}
\begin{proof}
    Since $\mu$ is compactly supported on $\R^{dT}$, there exists $R > 0$ 
    such that $\mu$ is supported on $[-R,R]^{dT}$. Then, by scaling $[-R,R]^{dT}$ into $[0,1]^{dT}$, we can apply Theorem \ref{thm:asconvergence_cube} to $\mu$ and complete the proof.
\end{proof}
\begin{proof}[Proof of Theorem \ref{thm:asconvergence}] 
    \hypertarget{proof:asconvergence}{\,\!} 
    The idea of the proof is to construct a measure $\nu \in \calP(\R^{dT})$ that is compactly supported so that we can apply Lemma~\ref{lem:asconvergence_compact}, while still very close to $\mu$ under adapted Wasserstein distance.\\
    \textbf{Step 1: }Since $\mu$ is integrable by assumption, for all $\epsilon > 0$ there exists a large enough $R_{\epsilon} = 2^{j_{\epsilon}}$ for some $j_{\epsilon}\geq 0$ s.t. the 
    \add{cube $K_{\epsilon} \coloneqq [-R_{\epsilon}, R_{\epsilon}]^{dT} \subseteq\R^{dT}$} satisfies $\int_{K_\epsilon^c}\Vert x \Vert \mu(dx) < \epsilon$ and $\mu(K_\epsilon^c) < \epsilon$. \add{For all $t=1,\dots,T$, let $K_{\epsilon,1:t} = [-R_\epsilon,R_\epsilon]^{dt}$, $K_{\epsilon, 1} = K_{\epsilon,1:1}$, and consider $x^{\epsilon} \in K_\epsilon^c$, which will be specified later.} For $x_1 \in \R^{d}$ and $x\in \R^{dT}$, we define $\phi_1 \colon \R^{d} \to \R^{d}$ and $\phi\colon \R^{dT} \to \R^{dT}$ by 
    \begin{equation*}
        \phi_1(x_1) = \left\{\begin{aligned}
            x_1, \quad &x_1 \in K_{\epsilon,1} \\
            x^\epsilon_1, \quad &x_{1}\not\in K_{\epsilon,1} \\
        \end{aligned}\right. ,\quad
        \phi(x) = \left\{\begin{aligned}
            &x, &x \in K_{\epsilon} \\
            &x^\epsilon, &x\not\in K_{\epsilon} \\
        \end{aligned}\right. .
    \end{equation*}
    Let $\pi_1 = (\mathbf{id}, \phi_1)_{\#}\mu_1$ and define, for all $t = 1,\dots,T-1$ and $x_{1:t}, y_{1:t} \in \R^{dt}$, 
    \begin{equation*}
 \pi_{x_{1:t},y_{1:t}} = \left\{\begin{aligned}
(\mathbf{id}, \phi_1)_{\#}\mu_{x_{1:t}}, \quad &x_{1:t} \in K_{\epsilon,1:t}, y_{1:t} = x_{1:t} \\
\mu_{x_{1:t}} \otimes \delta_{x_{t+1}^{\epsilon}}, \quad &x_{1:t}\not\in K_{\epsilon,1:t}, y_{1:t} = x^{\epsilon}_{1:t} \\
\end{aligned}\right. .
\end{equation*}
Further, denote $\pi = \pi_{1}\pi_{x_{1},y_{1}}\dots\pi_{x_{1:t},y_{1:t}}\dots\pi_{x_{1:T-1},y_{1:T-1}}$ and let $\nu$ be its second marginal. Notice that $\pi_{x_{1:t},y_{1:t}}(dx_{t+1}) = \mu_{x_{1:t}}(dx_{t+1})$ and 
    \begin{equation*}
        \pi_{x_{1:t},y_{1:t}}(dy_{t+1}) = \left\{\begin{aligned}
            {\phi_1}_{\#}\mu_{y_{1:t}}, \quad & y_{1:t} \in K_{\epsilon,1:t} \\
            \delta_{x_{t+1}^{\epsilon}}, \quad &y_{1:t}\not\in K_{\epsilon,1:t}\\
        \end{aligned}\right. ,
    \end{equation*}
    which implies that $\pi_{x_{1:t},y_{1:t}} \in \cpl(\mu_{x_{1:t}},\nu_{y_{1:t}})$, and thus $\pi\in\bccpl(\mu,\nu)$. Moreover, it is easy to check that $\nu = \phi_{\#}\mu$. Intuitively, we define $\pi$ in a dynamic way to transport $\mu$ to $\phi_{\#}\mu$. By construction, $\nu$ is compactly supported. Therefore, by Lemma~\ref{lem:asconvergence_compact}, 
    \begin{equation}
    \label{eq:thm:asconvergence:1}
        \limsup_{N \to \infty}\AWD(\nu,\bar{\nu}^{N}) = 0.
    \end{equation}
    Next, we estimate $\AWD(\mu,\nu)$.  Since we have already defined a bi-causal coupling between $\mu$ and $\nu$, that is $\pi = (\mathbf{id}, \phi)_{\#}\mu \in \bccpl(\mu,\nu)$, by the definition of adapted Wasserstein distance we have
    \begin{equation}\label{eq:thm:asconvergence:1.5}
    \begin{split}
        \AWD(\mu,\nu) \leq \int \Vert x - \phi(x) \Vert \mu(dx) &\leq \int_{K_\epsilon^c}\Big(\Vert x \Vert + \Vert \phi(x) \Vert\Big) \mu(dx).
    \end{split}
    \end{equation}
    \begin{figure}[ht]
        \centering
        \includegraphics[width=0.7\linewidth]{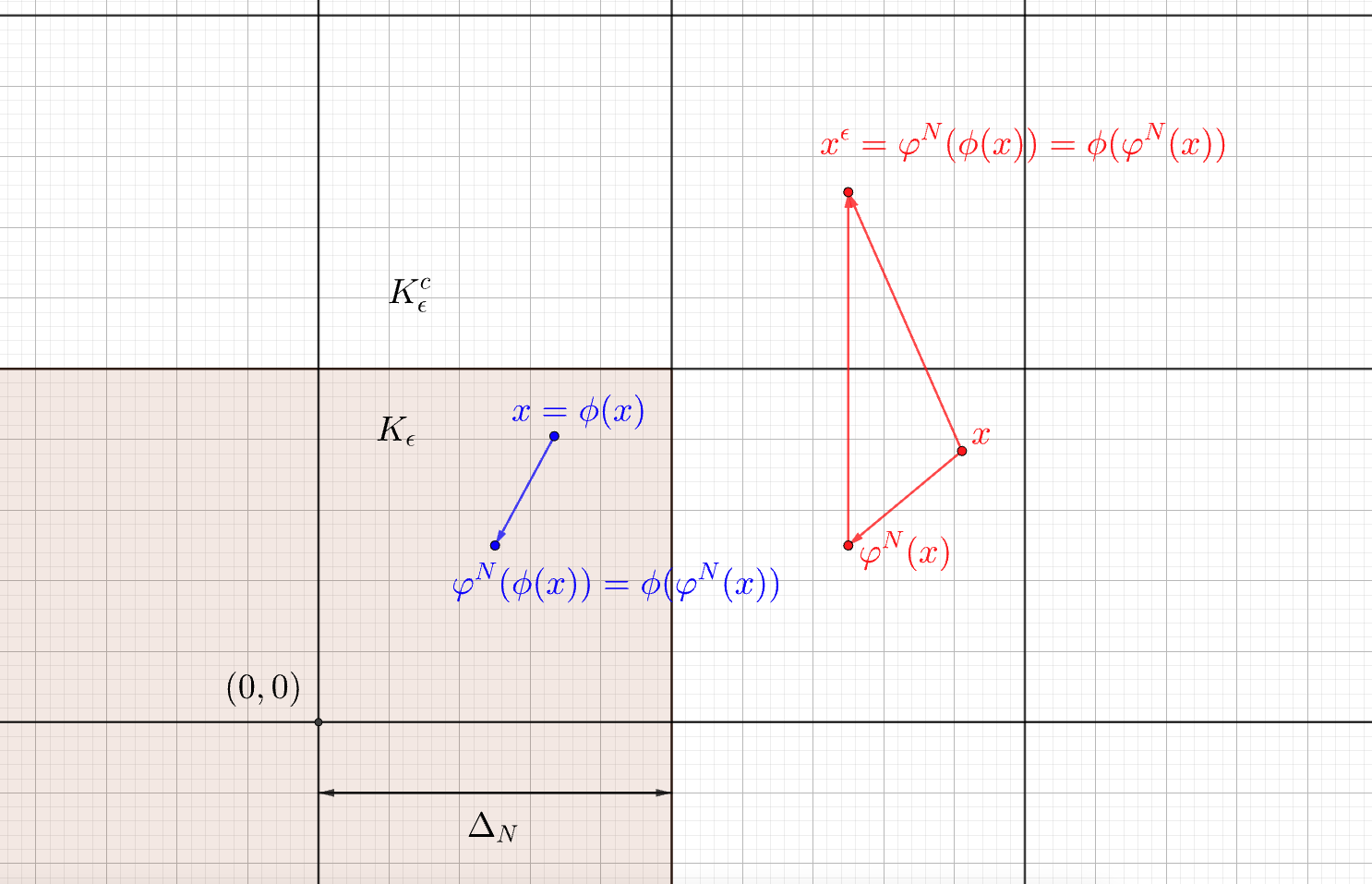}
        \caption{Visualization of $\varphi^N \circ \phi = \phi \circ \varphi^N$.}
        \label{fig:as}
    \end{figure}
    We choose $x^\epsilon = (2^{j_{\epsilon}} + \frac{1}{m},\dots, 2^{j_{\epsilon}} + \frac{1}{m})$ (resp. $x^\epsilon = (2^{j_{\epsilon}} + \frac{2^{j_\epsilon-1}}{m},\dots,2^{j_{\epsilon}} + \frac{2^{j_\epsilon-1}}{m})$) in the case of uniform (resp. non-uniform) grid, where $m = \lceil \frac{1}{\Delta_N}\rceil$, see Figure~\ref{fig:as}. Therefore, we have 
    \begin{equation*}
        \left\{\begin{aligned}
            &\Vert \phi(x) \Vert = \Vert x\Vert \leq 2T\sqrt{d}\Vert x \Vert,\quad &x \in K_\epsilon\\
            &\Vert \phi(x) \Vert = \Vert x^\epsilon\Vert \leq T\sqrt{d}\,\frac{2^j + \frac{2^{j-1}}{m}}{2^j}\Vert x \Vert \leq 2T\sqrt{d}\Vert x \Vert,\quad &x \in K_\epsilon^c
        \end{aligned}\right.  .
    \end{equation*}
    Together with \eqref{eq:thm:asconvergence:1.5}, this yields 
    \begin{equation}\label{eq:thm:asconvergence:2}
        \AWD(\mu,\nu) \leq \int_{K_\epsilon^c}\add{(1+2T\sqrt{d})}\Vert x \Vert \mu(dx) \leq \add{(1+2T\sqrt{d})}\epsilon.
    \end{equation}
    By the triangle equality, we see that
    \begin{equation}\label{eq:thm:asconvergence:triangle}
        \AWD(\mu,\bar{\mu}^{N}) \leq \AWD(\mu,\nu) +  \AWD(\nu,\bar{\nu}^{N})  + \AWD(\bar{\mu}^{N},\bar{\nu}^{N}).
    \end{equation}
    \add{Since $\AWD(\mu,\nu)$ and $\AWD(\nu,\bar{\nu}^{N})$ are already estimated, }the only term that remains to be estimated is $\AWD(\bar{\mu}^{N},\bar{\nu}^{N})$. For this purpose, we next define a bi-causal coupling $\tilde{\pi}$ between $\bar{\mu}^{N}_1$ and $\bar{\nu}^{N}_1$. 
\add{Notice that $\bar{\mu}^{N} = \varphi^{N}_{\#}\mu^{N}$ and $\bar{\nu}^{N} = \varphi^{N}_{\#}\phi_{\#}\mu^{N}$. By choosing $R_{\epsilon}$ to match the boundary of the grid s.t. $\phi\circ\varphi^{N} = \varphi^{N}\circ\phi$, we have $\bar{\nu}^{N} = \phi_{\#}\varphi^{N}_{\#}\mu^{N} = \phi_{\#}\bar{\mu}^{N}$, see Figure~\ref{fig:as}. So, by the same argument used above, we have that $\tilde{\pi} = (\mathbf{id}, \phi)_{\#}\bar{\mu}^{N}\in \bccpl(\bar{\mu}^{N}, \bar{\nu}^{N})$.} Therefore, by the definition of adapted Wasserstein distance,
    \begin{equation}\label{eq:thm:asconvergence:3.2}
    \begin{split}
        \AWD(\bar{\mu}^{N}, \bar{\nu}^{N}) &\leq \int \Vert \varphi^{N}(x) -  \varphi^{N}\circ\phi (x)\Vert \mu^{N}(dx) = \int_{K_{\epsilon}^{c}} \Vert \varphi^{N}(x) -  \varphi^{N}\circ\phi (x)\Vert \mu^{N}(dx)\quad (\phi\vert_{K_{\epsilon}} = \mathbf{id})\\
        &\leq \int_{K_{\epsilon}^{c}}\big( \Vert \varphi^{N}(x)\Vert  +  \Vert \varphi^{N}(x^{\epsilon})\Vert \big)\mu^{N}(dx) = \int_{K_{\epsilon}^{c}}\big( \Vert \varphi^{N}(x)\Vert  +  \Vert x^{\epsilon}\Vert \big)\mu^{N}(dx)\\ 
        &\leq \int_{K_{\epsilon}^{c}} 
        \add{(1+2T\sqrt{d})}\Vert \varphi^{N}(x)\Vert \mu^{N}(dx)= \add{(1+2T\sqrt{d})}\int_{K_\epsilon^c}\Vert x \Vert \bar{\mu}^{N}(dx).
    \end{split}
    \end{equation}
    \textbf{Step 2: } Now we claim that
    \begin{equation}\label{eq:thm:asconvergence:4}
        \begin{split}
            \limsup_{N \to \infty} \int_{K_\epsilon^c}\Vert x \Vert \bar{\mu}^{N}(dx) \leq \limsup_{N \to \infty} \int_{K_\epsilon^c}\Vert x \Vert \mu^{N}(dx).
        \end{split}
    \end{equation}
    We are going to prove this
    separately for the uniform and non-uniform adapted empirical measures. \\
    (i) In the case of uniform adapted empirical measure, for all $x \in \R^{dT}$ we have $\Vert \hat{\varphi}^{N}(x) - x\Vert \leq T\sqrt{d}\Delta_N$. Thus, we obtain that 
    \begin{equation}\label{eq:thm:asconvergence:5}
            \limsup_{N \to \infty} \int_{K_\epsilon^c}\Vert x \Vert \hat{\mu}^{N}(dx) \leq \limsup_{N \to \infty}\Big(\int_{K_\epsilon^c}\Vert x \Vert \mu^{N}(dx) +  T\sqrt{d}\Delta_N \Big)
        = \limsup_{N \to \infty} \int_{K_\epsilon^c}\Vert x \Vert \mu^{N}(dx).
    \end{equation}
    (ii) In the case of non-uniform adapted empirical measure, for all $j\in \N$ and $x \in \mathcal{A}_{j}$ we have $\Vert \check{\varphi}^{N}(x) - x\Vert \leq T\sqrt{d}\Delta_N 2^{j}$. Thus, we obtain that 
    \begin{equation*}
        \begin{aligned}
            \int_{K_\epsilon^c}\Vert x \Vert \check{\mu}^{N}(dx) &= \sum_{j=0}^{\infty}\int_{\mathcal{A}_{j} \cap K_\epsilon^c}\Vert \check{\varphi}^N(x) \Vert \mu^{N}(dx) \leq \sum_{j=0}^{\infty}\int_{\mathcal{A}_{j} \cap K_\epsilon^c}\big( \Vert \check{\varphi}^{N}(x) - x\Vert + \Vert x \Vert \big) \mu^{N}(dx)\\
            &\leq T\sqrt{d}\sum_{j=0}^{\infty}2^{j}\Delta_{N}\mu^{N}(\mathcal{A}_{j}\cap K_\epsilon^c) + \int_{K_\epsilon^c}\Vert x \Vert \mu^{N}(dx)\\
            &= T\sqrt{d}\Delta_N + 2T\sqrt{d}\Delta_N\sum_{j=1}^{\infty}2^{j-1}\mu^{N}(\mathcal{A}_{j}\cap K_\epsilon^c) + \int_{K_\epsilon^c}\Vert x \Vert \mu^{N}(dx),~\text{(Since for all $x \in \mathcal{A}_{j}$, $\Vert x \Vert \geq 2^{j-1}$)}\\
            &\leq T\sqrt{d}\Delta_N + 2T\sqrt{d}\Delta_N\sum_{j=1}^{\infty}\int_{\mathcal{A}_{j}\cap K_\epsilon^c}\Vert x \Vert \mu^{N}(dx)  + \int_{K_\epsilon^c}\Vert x \Vert \mu^{N}(dx)\\
            &= T\sqrt{d}\Delta_N + (2T\sqrt{d}\Delta_{N} + 1 )\int_{ K_\epsilon^c}\Vert x \Vert \mu^{N}(dx).
        \end{aligned}
    \end{equation*}
    By letting $N \to \infty$, $\Delta_N \to 0$ and we get 
    \begin{equation*}
        \limsup_{N \to \infty} \int_{K_\epsilon^c}\Vert x \Vert \check{\mu}^{N}(dx) \leq \limsup_{N \to \infty} \int_{K_\epsilon^c}\Vert x \Vert \mu^{N}(dx).
    \end{equation*}
    This combined with \eqref{eq:thm:asconvergence:5} completes the proof of the claim \eqref{eq:thm:asconvergence:4}.\\
    \textbf{Step 3: } Recall that $\mu$ is integrable by assumption. Then, by the law of large numbers,
    \begin{equation}\label{eq:thm:asconvergence:7}
        \begin{split}
            \lim_{N \to \infty} \int_{K_\epsilon^c}\Vert x \Vert \mu^{N}(dx) = \int_{K_\epsilon^c}\Vert x \Vert \mu(dx) \leq \epsilon.
        \end{split}
    \end{equation}
    Therefore, combining \eqref{eq:thm:asconvergence:3.2}, \eqref{eq:thm:asconvergence:4} and \eqref{eq:thm:asconvergence:7}, we have that
    \begin{equation}\label{eq:thm:asconvergence:8}
        \limsup_{N \to \infty}\AWD(\bar{\mu}^{N},\bar{\nu}^{N}) \leq \add{(1+2T\sqrt{d})}\epsilon.
    \end{equation}
    Finally, combining \eqref{eq:thm:asconvergence:1}, \eqref{eq:thm:asconvergence:2}, \eqref{eq:thm:asconvergence:triangle} and \eqref{eq:thm:asconvergence:8}, we conclude that 
    \[
        \limsup_{N \to \infty}\AWD(\mu,\bar{\mu}^{N})  \leq \limsup_{N \to \infty}\Big(\AWD(\mu,\nu) +  \AWD(\nu,\bar{\nu}^{N})  + \AWD(\bar{\mu}^{N},\bar{\nu}^{N})\Big)\leq \add{(2+4T\sqrt{d})}\epsilon.
        \]
  We then let $\epsilon \to 0$ and complete the proof of Theorem \ref{thm:asconvergence}. 
\end{proof}

\printbibliography
\end{document}